\documentclass[12pt]{amsart}
\usepackage{enumerate}
\usepackage{mathdots,bm,url,a4wide}
\usepackage{amsfonts,amssymb,mathscinet}
\usepackage[all]{xy}
\usepackage[stable]{footmisc}

\author{Erez Lapid}
\address{Department of Mathematics, Weizmann Institute of Science, Rehovot 7610001 Israel}
\email{erez.m.lapid@gmail.com}
\author{Zhengyu Mao}
\address{Department of Mathematics and Computer Science, Rutgers University, Newark, NJ 07102, USA}
\email{zmao@rutgers.edu}
\title[On an analogue of the Ichino--Ikeda conjecture]{On an analogue of the Ichino--Ikeda conjecture for Whittaker coefficients on the metaplectic group}
\date{\today}
\thanks{E.L. was partially supported by Grant \# 711733 from the Minerva Foundation.}
\thanks{Z.M. was partially supported by NSF grants DMS 1000636 and 1400063 and by a fellowship from the Simons Foundation}

\newcommand{\A}{\mathbb{A}}                            
\newcommand{\C}{\mathbb{C}}                            
\newcommand{\R}{\mathbb{R}}                            
\newcommand{\Z}{\mathbb{Z}}                            
\newcommand{\bs}{\backslash}
\renewcommand{\Re}{\operatorname{Re}}
\newcommand{\symp}{\mathbb{V}}
\newcommand{\Hom}{\operatorname{Hom}}
\newcommand{\vol}{\operatorname{vol}}
\newcommand{\temp}{\operatorname{temp}}
\newcommand{\dsum}{\oplus}
\newcommand{\iii}{\mathrm{i}}
\newcommand{\OO}{\mathcal{O}}                         
\newcommand{\eps}{\epsilon}
\newcommand{\auxx}{{\bm \epsilon}}
\newcommand{\auxthree}{\auxx_1}
\newcommand{\auxauxone}{\auxx_2}
\newcommand{\auxone}{\auxx_3}
\newcommand{\auxtwo}{\auxx_4}
\newcommand{\auxfour}{\auxx_5}
\newcommand{\stint}{\int^{st}}                         
\newcommand{\des}{\mathcal{D}_\psi}                     
\newcommand{\desinv}{\mathcal{D}_{\psi^{-1}}}
\newcommand{\spcl}{^\circ}                               
\newcommand{\whitform}{A^{\psi}}                              

\newcommand{\nwhitform}{A^\psi_\sharp}
\newcommand{\nwhitformd}{A^{\psi^{-1}}_\sharp}
\newcommand{\Mint}{A_e^{\psi}}                                
\newcommand{\Mintd}{A_e^{\psi^{-1}}}
\newcommand{\Bes}{\mathbb{B}}                                 
\newcommand{\smth}{\operatorname{sm}}
\newcommand{\wgtt}[1]{\nu(#1)}
\newcommand{\wgt}[1]{\nu'(#1)}
\newcommand{\good}{good}
\newcommand{\zigzag}{\mathcal{Z}}
\newcommand{\per}[1]{\mathfrak{P}^{#1}}
\newcommand{\Etwon}{\mathfrak{E}}
\newcommand{\soment}{T''}
\newcommand{\factor}{\Delta}
\newcommand{\Imk}{T}
\newcommand{\spcltrs}{S}
\newcommand{\twU}{\animg{w'_{U'}}}
\newcommand{\mnrs}{\Delta}

\renewcommand{\d}[1]{#1^{\vee}}
\newcommand{\alt}[1]{#1^{\wedge}}
\newcommand{\Bil}{B}                         
\newcommand{\BilM}{D}                         
\newcommand{\newBil}{\underline{B}}
\newcommand{\oldhalf}{Y^{\psi}}                      
\newcommand{\oldhalfd}{Y^{\psi^{-1}}}
\newcommand{\newhalf}{E^{\psi}}                      
\newcommand{\newhalfd}{E^{\psi^{-1}}}
\newcommand{\animg}[1]{\widetilde{#1}}                        
\newcommand{\Mp}{\widetilde{\operatorname{Sp}}}
\newcommand{\Mat}{\operatorname{Mat}}
\newcommand{\zerocol}{J}
\newcommand{\GL}{\operatorname{GL}}
\newcommand{\Sp}{\operatorname{Sp}}
\newcommand{\SO}{\operatorname{SO}}
\newcommand{\mira}{\mathcal{P}}                     
\newcommand{\sym}{\operatorname{sym}}
\newcommand{\Ind}{\operatorname{Ind}}                   
\newcommand{\res}{\operatorname{res}}                   
\newcommand{\Levi}{M}
\newcommand{\GLnn}{{\mathbb M}}
\newcommand{\GLn}{{{\mathbb M}'}}
\newcommand{\semipair}[2]{\{#1,#2\}}

\newcommand{\alg}[1]{\bf #1}
\newcommand{\csgr}{\mathcal{CSGR}}                     
\newcommand{\Irr}{\operatorname{Irr}}                  
\newcommand{\rest}{\big|}                              
\newcommand{\Cusp}{\operatorname{Cusp}}                
\newcommand{\Mcusp}{\operatorname{MCusp}}              
\newcommand{\meta}{\operatorname{meta}}                
\newcommand{\gen}{\operatorname{gen}}                  
\newcommand{\sqr}{\operatorname{sqr}}                  
\newcommand{\levi}{\varrho}                            
\newcommand{\toG}{\eta}                                
\newcommand{\toM}{\eta_{\GLnn}}                        
\newcommand{\toLevi}{\eta_{\Levi}}                     
\newcommand{\toMd}{\eta_{\GLnn}^\vee}                  
\newcommand{\toLevid}{\eta_{\Levi}^\vee}               
\newcommand{\toU}{\ell}
\newcommand{\startran}[1]{\breve{#1}}             
\newcommand{\swz}{\mathcal{S}}
\newcommand{\swrz}{C_c^{\infty}}
\newcommand{\modulus}{\delta}
\newcommand{\diag}{\operatorname{diag}}
\newcommand{\Vm}{V_-}
\newcommand{\Vp}{V_+}
\newcommand{\altV}{V^\sharp}
\newcommand{\Heise}{\mathcal{H}}        
\newcommand{\whit}{\mathcal{W}}                          
\newcommand{\Whit}{\mathbb{W}}                           
\newcommand{\WhitM}{\Whit^{\psi_{N_\GLnn}}}                  
\newcommand{\WhitMd}{\Whit^{\psi_{N_\GLnn}^{-1}}}
\newcommand{\WhitML}{\Whit^{\psi_{N_\Levi}}}                  
\newcommand{\WhitMLd}{\Whit^{\psi_{N_\Levi}^{-1}}}
\newcommand{\WhitG}{\Whit^{\psi_{\tilde N}}}                  
\newcommand{\WhitGd}{\Whit^{\psi_{\tilde N}^{-1}}}
\newcommand{\spclW}{\Ind(\WhitML(\pi))\spcl_\sharp}
\newcommand{\spclWd}[1]{\Ind(\WhitMLd(#1))\spcl_{\natural}}
\newcommand{\spclWdM}{\WhitMd(\d\pi)_{\natural}}
\newcommand{\symspace}{\mathfrak{s}}                 

\newcommand{\subUbar}{\bar U^\urcorner}
\newcommand{\oldsubUbar}{\hat{\psi}_{\subUbar}}
\newcommand{\vf}{N_{\GLnn}^{\flat}}
\newcommand{\vrq}{\toMd(N'_{\GLn})\bs\vf}
\newcommand{\vrbar}{\bar R}
\newcommand{\kgrp}{N^\sharp}
\newcommand{\kgrpq}{\altV_{\Levi}\bs\kgrp}
\newcommand{\bigvr}{\toLevi(N'_{\GLn})\bs\kgrp}
\newcommand{\vs}{V_\Delta}
\newcommand{\remr}{N_{\GLnn,\Delta}}
\newcommand{\one}{\epsilon}        
\newcommand{\num}{{\mathfrak a}}
\newcommand{\few}{{\hat w}}
\newcommand{\rkn}{\mathbf{n}}
\newcommand{\wnn}{{w_0^\GLnn}}
\newcommand{\wnnM}{{w_0^\Levi}}
\newcommand{\wn}{{w_0^{\GLn}}}
\newcommand{\wnM}{{w_0^{\Levi'}}}
\newcommand{\Hei}[1]{{#1}_{\Heise}}
\newcommand{\sprod}[2]{\left\langle#1,#2\right\rangle}
\newcommand{\abs}[1]{\left|{#1}\right|}
\newcommand{\sm}[4]{\left(\begin{smallmatrix}{#1}&{#2}\\{#3}&{#4}\end{smallmatrix}\right)}
\newcommand{\weil}{\omega}
\newcommand{\wev}{\weil_{\psi}}
\newcommand{\wevinv}{\weil_{\psi^{-1}}}
\newcommand{\weilfctr}{\gamma_\psi}                       
\newcommand{\der}{\operatorname{der}}                     
\newcommand{\spclt}{\mathfrak{t}}
\newcommand{\ddg}{\mathfrak{d}}

\newtheorem{theorem}{Theorem}[section]
\newtheorem{lemma}[theorem]{Lemma}
\newtheorem{proposition}[theorem]{Proposition}
\newtheorem{remark}[theorem]{Remark}
\newtheorem{conjecture}[theorem]{Conjecture}
\newtheorem{definition}[theorem]{Definition}
\newtheorem{corollary}[theorem]{Corollary}

\numberwithin{equation}{section}

\begin{document}

\begin{abstract}
In previous papers we formulated an analogue of the Ichino--Ikeda conjectures for Whittaker--Fourier coefficients of automorphic forms
on quasi-split classical groups and the metaplectic group of arbitrary rank. In the latter case we reduced the conjecture to a local identity.
In this paper we will prove the local identity in the $p$-adic case, and hence the global conjecture under simplifying conditions
at the archimedean places.
\end{abstract}

\maketitle

\setcounter{tocdepth}{1}
\tableofcontents

\section{Introduction}

Let $G$ be a quasi-split group over a number field $F$ with ring of adeles $\A$.
In a previous paper \cite{MR3267120} we formulated (under some hypotheses) a conjecture relating Whittaker-Fourier coefficients of cusp forms on $G(F)\bs G(\A)$
to the Petersson inner product.
This conjecture is in the spirit of conjectures of Sakellaridis--Venkatesh \cite{1203.0039} and Ichino--Ikeda \cite{MR2585578},
which attempt to generalize the classical work of Waldspurger \cite{MR783511, MR646366}.
In the case of (quasi-split) classical groups, as well as the metaplectic group (i.e.,
the metaplectic double cover of the symplectic group) we explicated this conjecture
using the descent construction of Ginzburg--Rallis--Soudry \cite{MR2848523}
and the functorial transfer of generic representations of classical groups by Cogdell--Kim--Piatetski-Shapiro--Shahidi
\cite{MR1863734, MR2075885, MR2767514}.

In a follow-up paper \cite{1401.0198} we reduced the global conjecture in the metaplectic case to a local conjectural identity.
We also gave a \emph{purely formal} argument for the case of $\Mp_1$ (i.e., ignoring convergence issues).
In this paper we will prove the local identity in the $p$-adic case, justifying the heuristic analysis
(and extending it to the general case).

Let us recall the conjecture of \cite{MR3267120} in the case of the metaplectic group $\Mp_n(\A)$,
the double cover of $\Sp_n(\A)$ with the standard (Rao) co-cycle.
We view $\Sp_n(F)$ as a subgroup of $\Mp_n(\A)$.
For any genuine function $\varphi$ on $\Sp_n(F)\bs\Mp_n(\A)$ we consider the Whittaker coefficient
\[
\tilde\whit(\tilde\varphi)=\tilde\whit^{\psi_{\tilde N}}(\tilde\varphi):=
(\vol(N'(F)\bs N'(\A)))^{-1}\int_{N'(F)\bs N'(\A)}\tilde\varphi(u)\psi_{\tilde N}(u)^{-1}\ du.
\]
Here $\psi_{\tilde N}$ is a non-degenerate character on $N'(\A)$, trivial on $N'(F)$
where $N'$ is the standard maximal unipotent subgroup of $\Sp_n$.
(We view $N'(\A)$ as a subgroup of $\Mp_n(\A)$.)
We also consider the inner product
\[
(\tilde\varphi,\tilde\varphi^\vee)_{\Sp_n(F)\bs\Sp_n(\A)}=(\vol(\Sp_n(F)\bs\Sp_n(\A)))^{-1}
\int_{\Sp_n(F)\bs\Sp_n(\A)}\tilde\varphi(g)\tilde\varphi^\vee(g)\ dg.
\]
of two square-integrable genuine functions on $\Sp_n(F)\bs\Mp_n(\A)$.

Another ingredient in the conjecture of \cite{MR3267120} is a regularized integral
\[
\stint_{N'(F_S)}f(u)\ du
\]
for a finite set of places $S$ and for a suitable class of smooth functions $f$ on $N'(F_S)$.
Suffice it to say that if $S$ consists solely of non-archimedean places then
\[
\stint_{N'(F_S)}f(u)\ du=\int_{N_1'}f(u)\ du
\]
for any sufficiently large compact open subgroup $N_1'$ of $N'(F_S)$.
(The definition of the regularized integral is different in the archimedean case, however in this paper we do not use regularized integrals over archimedean fields.)

The conjecture of \cite{MR3267120} is applicable for $\psi_{\tilde N}$-generic representations which are not exceptional,
in the sense that their theta $\psi$-lift to $\SO(2n+1)$ is cuspidal (or equivalently, their theta $\psi$-lift to $\SO(2n-1)$ vanishes;
here $\psi$ is determined by $\psi_{\tilde N}$).
By \cite[Chapter~11]{MR2848523} which is also based on \cite{MR1863734}, there is a one-to-one correspondence between these representations
and automorphic representations $\pi$ of $\GL_{2n}(\A)$ which are the isobaric sum $\pi_1\boxplus\dots\boxplus\pi_k$ of
pairwise inequivalent irreducible cuspidal representations $\pi_i$ of
$\GL_{2n_i}(\A)$, $i=1,\dots,k$ (with $n_1+\dots+n_k=n$) such that $L^S(\frac12,\pi_i)\ne0$ and $L^S(s,\pi_i,\wedge^2)$ has a pole (necessarily simple)
at $s=1$ for all $i$. Here $L^S(s,\pi_i)$ and $L^S(s,\pi_i,\wedge^2)$ are the standard and exterior square (partial) $L$-functions, respectively.
More specifically, to any such $\pi$ one constructs a $\psi_{\tilde N}$-generic representation $\tilde\pi$ of $\Mp_n(\A)$,
which is called the $\psi_{\tilde N}$-descent of $\pi$.
The theta $\psi$-lift of $\tilde\pi$ is the unique irreducible generic cuspidal representation of $\SO(2n+1)$ which lifts to $\pi$.

\begin{conjecture}(\cite[Conjecture 1.3]{MR3267120}) \label{conj: metplectic global}
Assume that $\tilde\pi$ is the $\psi_{\tilde N}$-descent of $\pi$ as above.
Then for any $\tilde\varphi\in\tilde\pi$ and $\d{\tilde\varphi}\in\d{\tilde\pi}$ and for any sufficiently large finite set $S$ of places of $F$ we have
\begin{multline} \label{eq: globalidentity}
\tilde\whit^{\psi_{\tilde N}}(\tilde\varphi)\tilde\whit^{\psi_{\tilde N}^{-1}}(\d{\tilde\varphi})=
2^{-k}(\prod_{i=1}^n\zeta_F^S(2i))\frac{L^S(\frac12,\pi)}{L^S(1,\pi,\sym^2)}\times\\
(\vol(N'(\OO_S)\bs N'(F_S)))^{-1}\stint_{N'(F_S)}(\tilde\pi(u)\tilde\varphi,\d{\tilde\varphi})_{\Sp_n(F)\bs\Sp_n(\A)}\psi_{\tilde N}(u)^{-1}\ du
\end{multline}
Here $\zeta_F^S(s)$ is the partial Dedekind zeta function and $\OO_S$ is the ring of $S$-integers of $F$
and $L^S(s,\pi,\sym^2)$ is the symmetric square partial $L$-function of $\pi$.
\end{conjecture}

The main result in \cite{1401.0198} is
\begin{theorem}(\cite[Theorem 6.2]{1401.0198}) \label{thm: metplectic global}
In the above setup we have
\begin{multline} \label{eq: globalidentity2}
\tilde\whit^{\psi_{\tilde N}}(\tilde\varphi)\tilde\whit^{\psi_{\tilde N}^{-1}}(\d{\tilde\varphi})=
2^{-k}(\prod_{i=1}^n\zeta_F^S(2i))\frac{L^S(\frac12,\pi)}{L^S(1,\pi,\sym^2)}\times\\
\big(\prod_{v\in S}c_{\pi_v}^{-1}\big)(\vol(N'(\OO_S)\bs N'(F_S)))^{-1}
\stint_{N'(F_S)}(\tilde\pi(u)\tilde\varphi,\d{\tilde\varphi})_{\Sp_n(F)\bs\Sp_n(\A)}\psi_{\tilde N}(u)^{-1}\ du
\end{multline}
where $c_{\pi_v}$, $v\in S$ are certain non-zero constants which depend only on the local representations $\pi_v$.
\end{theorem}

The main result of this paper is
\begin{theorem}\label{thm: intromain}
In Theorem \ref{thm: metplectic global} we have $c_{\pi_v}=\epsilon(\frac12,\pi_v,\psi_v)$ (the root number of $\pi_v)$ for all finite places $v$.
\end{theorem}

We also show in Proposition~\ref{prop: central} that the root number of $\pi_v$ equals the central sign of $\tilde\pi_v$.
In \cite{1404.2909} it is shown that Theorem \ref{thm: intromain} implies the formal degree conjecture of
Hiraga--Ichino--Ikeda \cite{MR2350057} (or more precisely, its metaplectic analogue) for generic square-integrable
representations of $\Mp_n$. Conversely, in the real case (where the formal degree conjecture is a reformulation of classical results
of Harish-Chandra) it is shown that $c_{\pi_v}=\epsilon(\frac12,\pi_v,\psi_v)$ if $\tilde\pi_v$ is square-integrable.
(Note that in the square-integrable case, matrix coefficients are integrable on $N'(F_v)$ and no regularization is necessary.) We conclude:

\begin{corollary}
Conjecture~\ref{conj: metplectic global} holds if $F$ is totally real and $\tilde\pi_\infty$ is discrete series.
\end{corollary}

Theorem \ref{thm: intromain} is the culmination of the series of papers \cite{MR3220931}, \cite{1401.0198} and \cite{MR3431601}.
More precisely, the theorem can be formulated as an identity -- the Main Identity \eqref{eq: MI'} explicated in
\S\ref{sec: firstred}, (based on \cite{1401.0198}) between integrals of Whittaker functions in the induced spaces of $\pi$ (in a local setting).
In principle, formal manipulations using the functional equations of \cite{MR3220931} reduce the identity to the results of \cite{MR3431601}.
Such an argument was described heuristically for the case $n=1$ in \cite[\S7]{1401.0198}.
However, making this rigorous (even in the case $n=1$ and for $\pi$ supercuspidal) seems non-trivial because the integrals
only converge as iterated integrals. This is the main task of the present paper.

In \S\ref{sec: firstred} we reduce the theorem to the cases where $\pi$ is tempered and satisfies some good properties.
Here we rely on the classification, due to Matringe \cite{MR3430877}, of the generic representations admitting a non-trivial
$\GL_n\times\GL_n$-invariant functionals.
We also use a globalization result (\cite[Appendix A]{1404.2909}) which is based on a result of Sakellaridis-Venkatesh \cite{1203.0039}.

The rest of the argument is purely local.
In \S\ref{sec: preliminary} we start the manipulation of the left-hand side of the Main Identity.
It is technically important to restrict oneself to certain special sections in the induced space.
This is possible by a nonvanishing result on the Bessel function of generic representations proved in Appendix \ref{sec: nonvanishing}.
Another useful idea is to write the left-hand side of the Main Identity (for $W$ special) as
$\Bil(W,M(\pi,\frac12)\alt{W},\frac12)$ where
$\Bil(W,\d{W},s)$, $s\in\C$ is an analytic family of bilinear forms on $I(\pi,s)\times I(\d\pi,-s)$, and $M(\pi,s)$ is the intertwining operator on $I(\pi,s)$.
This relies on results of Baruch \cite{MR2192818}, generalized to the present context in \cite{MR3021791}.

The reason for introducing this analytic family is that because of convergence issues, we can only apply the functional equations of \cite{MR3220931}
for $\Re s\ll0$. This is the most delicate step, which is described in \S\ref{sec: FEnewBil}. It entails a further restriction on $\alt{W}$
(which is fortunately harmless for our purpose). The upshot is an expression (for special $W$ and $\alt{W}$ and for $\Re s\ll0$)
\[
\Bil(W,M(\pi,s)\alt{W},s)=\int\newhalf(M^*_sW,t)\newhalfd(\alt{W}_s,t)\,\frac{dt}{\abs{\det t}}
\]
where $t$ is integrated over a certain $n$-dimensional torus and $\newhalf$ is a certain integral of $W$.
The restriction on $\alt{W}$ ensures that the integrand is compactly supported.
Moreover, by \cite{MR3431601}, as a distribution in $t$, $\newhalf(W_s,t)$ extends to an entire function on $s$.
Therefore the above identity is meaningful for all $s$ where $M(\pi,s)$ is holomorphic.
Specializing to $s=\frac12$, the remaining assertions are that $\newhalf(M^*W,t)$
is constant in $t$ (and in particular, is a function) whose value can be explicated, while the integral of
$\newhalfd(\alt{W}_{\frac12},t)$ factors through $M(\pi,\frac12)\alt{W}$, again in an explicit way.
This is the content of Corollaries \ref{cor: tindep} and \ref{cor: Mfactor} respectively, which
follow from the representation-theoretic results established in \cite{MR3431601}.
We refer the reader to \S\ref{sec: nsketch} below for a more detailed sktech of the proof.

\subsection{Acknowledgement}
Part of this work was done while the authors were hosted by the Erwin Schr\"odinger Institute in Vienna
for the program ``Research in Teams''.
We are grateful to the ESI, and especially Joachim Schwermer, for their hospitality.
We also thank Atsushi Ichino and Marko Tadi\'c for helpful correspondence.

\section{Notation and preliminaries} \label{sec: notation}
For the convenience of the reader we introduce in this section the most common notation
that will be used throughout.

We fix a positive integer $\rkn$  (not to be confused with a running variable $n$). In the following $m$ is either $\rkn$ or $2\rkn$.

Let $F$ be a local field of characteristic $0$.

\subsection{Groups, homomorphisms and group elements} \label{sec: elements}

All algebraic groups are defined over $F$.
We typically denote algebraic varieties (or groups) over $F$ by boldface letters (e.g., $\alg{X}$) and denote their set (or group) of $F$-points
by the corresponding plain letter (e.g., $X$). (In most cases $\alg{X}$ will be clear from the context.)

\begin{itemize}

\item $I_m$ is the identity matrix in $\GL_m$, $w_m$ is the $m\times m$-matrix with ones on
the non-principal diagonal and zeros elsewhere.

\item For any group $Q$, $Z_Q$ is the center of $Q$; $e$ is the identity element of $Q$.
We denote the modulus function of $Q$ (i.e., the quotient of a right Haar measure by a left Haar measure) by $\modulus_Q$.

\item $\Mat_m$ is the vector space of $m\times m$ matrices over $F$.

\item $x\mapsto x^t$ is the transpose on $\Mat_m$; $x\mapsto \startran{x}$ is the twisted transpose map on $\Mat_m$
given by $\startran{x}=w_m x^t w_m$; $g\mapsto g^*$ is the outer automorphism of $\GL_m$ given by $g^*=w_m^{-1}\,(g^t)^{-1} w_m$.

\item $\symspace_m=\{x\in\Mat_m:\startran{x}=x\}$.

\item $\GLnn=\GL_{2{\rkn}}$, $\GLn=\GL_{\rkn}$.

\item $G=\Sp_{2{\rkn}}=\{g\in\GL_{4{\rkn}}:\, g^t\sm{}{w_{2\rkn}}{-w_{2\rkn}}{}g=\sm{}{w_{2\rkn}}{-w_{2\rkn}}{}\}$.

\item $G'=\Sp_{\rkn}=\{g\in\GL_{2{\rkn}}:\,g^t\sm{}{w_\rkn}{-w_\rkn}{}g=\sm{}{w_\rkn}{-w_\rkn}{}\}$.

\item $G'$ is embedded as a subgroup of $G$ via $g\mapsto\toG(g)=\diag(I_{\rkn},g,I_{\rkn})$.

\item $P=\Levi\ltimes U$ (resp., $P'=\Levi'\ltimes U'$) is the Siegel parabolic subgroup of $G$ (resp., $G'$), with its standard Levi decomposition.

\item $\bar P=P^t$ is the opposite parabolic of $P$, with unipotent radical $\bar U=U^t$.

\item We use the isomorphism $\levi(g)=\diag(g,g^*)$ to identify $\GLnn$ with $\Levi\subset G$.
Similarly for $\levi':\GLn\rightarrow\Levi'\subset G'$.

\item We use the embeddings $\toM(g)=\diag(g,I_{\rkn})$ and $\toMd(g)=\diag(I_{\rkn},g)$ to identify $\GLn$ with subgroups of $\GLnn$.
We also set $\toLevi=\levi\circ\toM$ and $\toLevid=\levi\circ\toMd=\toG\circ\levi'$.

\item $K$ is the standard maximal compact subgroup of $G$. (In the $p$-adic case it consists of the matrices with integral entries.)

\item $N$  is the standard maximal unipotent subgroup of $G$  consisting of upper unitriangular matrices;
$T$  is the maximal torus of $G$  consisting of diagonal matrices;
$B=T\ltimes N$ is the Borel subgroup of $G$.

\item For any subgroup $X$ of $G$ we write $X'=\toG^{-1}(X)$, $X_\Levi=X\cap\Levi$ and $X_\GLnn=\levi^{-1}(X_\Levi)$;
similarly $X'_{\Levi'}=X'\cap \Levi'$ and $X'_\GLn=\levi'^{-1}(X'_{\Levi'})$.

\item $\toU_\GLnn:\Mat_\rkn\rightarrow N_\GLnn$ is the group embedding given by $\toU_\GLnn(x)=\sm{I_{{\rkn}}}{x}{}{I_{{\rkn}}}$
and $\toU_\Levi=\levi\circ\toU_\GLnn$.

\item $\toU:\symspace_{2{\rkn}}\rightarrow U$ is the isomorphism given by $\toU(x)=\sm{I_{2{\rkn}}}{x}{}{I_{2{\rkn}}}$.

\item $\tilde G=\Mp_{\rkn}$ is the metaplectic group, i.e., the non-trivial two-fold cover of $G'$ (unless $F$ is complex).
We write elements of $\tilde G$ as pairs $(g,\epsilon)$, $g\in G$, $\epsilon=\pm1$
where the multiplication is given by Rao's cocycle. (Cf. \cite{MR2848523}.)

\item When $g\in G'$, we write $\animg{g}=(g,1)\in\tilde G$. Of course, $g\mapsto\animg{g}$ is not a group homomorphism.
However, we do have $\animg{gn}=\animg{g}\animg{n}$ and $\animg{ng}=\animg{n}\animg{g}$ for any $g\in G'$, $n\in N'$.

\item $\tilde N$  (resp., $\tilde P$) is the inverse image of $N'$ (resp., $P'$) under the canonical projection $\tilde G\rightarrow G'$.
We will identify $N'$ with a subgroup of $\tilde N$ via $n\mapsto\animg{n}$.

\item $\xi_m=(0,\ldots,0,1)\in F^m$.

\item $\mira$ is the mirabolic subgroup of $\GLnn$ consisting of the elements $g$ such that $\xi_{2{\rkn}} g=\xi_{2{\rkn}}$.

\item $E=\diag(1,-1,\dots,1,-1)\in \GLnn$. $H$ is the centralizer of $\levi(E)$ in $G$. It is isomorphic to $\Sp_{\rkn}\times\Sp_{\rkn}$.

\item $H_\GLnn$ is then the centralizer of $E$ in $\GLnn$. It is isomorphic to $\GL_{\rkn}\times\GL_{\rkn}$.

\item $w_0'=\sm{}{w_{\rkn}}{-w_{\rkn}}{}\in G'$ represents the longest Weyl element of $G'$.

\item $w_U=\sm{}{I_{2{\rkn}}}{-I_{2{\rkn}}}{}\in G$ represents the longest $\Levi$-reduced Weyl element of $G$.

\item $w'_{U'}=\sm{}{I_{\rkn}}{-I_{\rkn}}{}\in G'$ represents the longest $\Levi'$-reduced Weyl element of $G'$.

\item $\wnn=w_{2{\rkn}}\in \GLnn$ represents the longest Weyl element of $\GLnn$; $\wnnM=\levi(\wnn)$.

\item $\wn=w_{\rkn}\in\GLn$ represents the longest Weyl element of $\GLn$; $\wnM=\levi'(\wn)$.

\item $w_{2{\rkn},{\rkn}}=\sm{}{I_{\rkn}}{I_{\rkn}}{}\in\GLnn$, $w_{2{\rkn},{\rkn}}'=\sm{}{I_{\rkn}}{\wn}{}\in\GLnn$.

\item $\gamma=w_U\toG(w'_{U'})^{-1}=\left(\begin{smallmatrix}&I_{\rkn}&&\\&&&I_{\rkn}\\-I_{\rkn}&&&\\&&I_{\rkn}&\end{smallmatrix}\right)\in G$.

\item $\ddg=\diag(1,-1,\ldots,(-1)^{\rkn-1})\in\Mat_\rkn$,
$\auxthree=\toU_\Levi((-1)^\rkn\ddg)\in N_\Levi$,
$\auxauxone=\toU_\GLnn(\ddg)\in N_\GLnn$,
$\auxone=w_{2\rkn,\rkn}'\auxauxone\in\GLnn$,
$\auxtwo=\toU_{\GLnn}(-\frac12\ddg\wn)\in N_\GLnn$.

\item $V$ (resp., $\altV$) is the unipotent radical of the standard parabolic subgroup of $G$ with Levi
$\GL_1^{\rkn}\times\Sp_{\rkn}$ (resp., $\GL_1^{{\rkn}-1}\times\Sp_{{\rkn}+1}$).
Thus, $N=\toG(N')\ltimes V$, $\altV$ is normal in $V$ and $V/\altV$ is isomorphic to the Heisenberg group
of dimension $2{\rkn}+1$ (see below). Also $V=V_\Levi\ltimes V_U$ where $V_U=V\cap U=
\{\toU(\sm{x}{y}{}{\startran{x}}):\ x\in \Mat_\rkn, y\in \symspace_{\rkn}\}$.

\item $\Vm=\altV_{\Levi}\ltimes V_U$. (Recall $\altV_{\Levi}=\altV\cap\Levi$ by our convention.)

\item $V_\gamma=V\cap\gamma^{-1}N\gamma=\toG(w'_{U'})V_{\Levi}\toG(w'_{U'})^{-1}=
\toLevi(N'_\GLn)\ltimes\{\toU(\sm{x}{}{}{\startran{x}}):x\in\Mat_\rkn\}\subset\Vm$.

\item $\Vp\subset V$ is the image under $\toU_\Levi$ of the space of ${\rkn}\times {\rkn}$-matrices
whose rows are zero except possibly for the last one.  Thus, $V=\Vp\ltimes \Vm$.
For $c=\toU_\Levi(x)\in \Vp$ we denote by $\underline c\in F^{\rkn}$ the last row of $x$.

\item $\kgrp=\Vm\rtimes\toG(N')$. It is the stabilizer in $N$ of the character $\psi_U$ defined below.

\item $\vf=(\kgrp_\GLnn)^*$.

\item $\zerocol$ is the subspace of $\Mat_\rkn$ consisting of the matrices whose first column is zero.

\item $\vrbar=\{\sm{I_{\rkn}}{}{x}{n^t}:x\in \zerocol, n\in N'_{\GLn}\}.$

\item $\soment=Z_\GLnn\times \toMd(T'_{\GLn})=\{\diag(t_1,\dots,t_{2\rkn}):t_1=\dots=t_{\rkn}\}\subset T_\GLnn$.

\end{itemize}

\subsection{Characters} \label{sec: characters}
We fix a non-trivial unitary character $\psi$ of $F$.
For each of the unipotent groups $X$ listed below we assign a character $\psi_X$ of $X$ as follows
\begin{gather*}
\psi_{N_\GLnn}(u)=\psi(u_{1,2}+\dots+u_{2{\rkn}-1,2{\rkn}})\text{ (non-degenerate)}\\
\psi_{N_\Levi}\circ\levi=\psi_{N_\GLnn}\\
\psi_{N'_\GLn}(u')=\psi(u'_{1,2}+\dots+u'_{{\rkn}-1,{\rkn}})\text{ (non-degenerate)}\\
\psi_{N'_{\Levi'}}\circ\levi'=\psi_{N'_{\GLn}}\text{ i.e., }\psi_{N'_{\Levi'}}(n)=\psi_{N_\Levi}(\gamma \toG(n)\gamma^{-1})\\
\psi_{N'}(nu)=\psi_{N'_{\Levi'}}(n)\psi(\frac12u_{{\rkn},{\rkn}+1})^{-1}, n\in N'_{\Levi'},\,u\in U'\\
\psi_{\tilde N}\text{ is the extension of $\psi_{N'}$ to a genuine character of $\tilde N$}\\
\psi_N(nu)=\psi_{N_\Levi}(n),\ \ \ n\in N_\Levi, u\in U\text{ (a degenerate character)}\\
\psi_U(\toU(v))=\psi(\frac12(v_{{\rkn},{\rkn}+1}-v_{2{\rkn},1}))\\
\psi_{\Vm}(vu)=\psi_{N_\Levi}(v)^{-1}\psi_U(u),\ \ v\in \altV_{\Levi}, u\in V_U\\
\psi_{\kgrp}(\toG(n)v)=\psi_{N'}(n)\psi_{\Vm}(v), \ \ n\in N',v\in\Vm\\
\psi_{\kgrp_\GLnn}=\psi_{\kgrp}\circ\levi\\
\psi_{\vf}(m)=\psi_{\kgrp_\GLnn}(m^*),\ m\in\vf\\
\psi_{\bar U}(\toU(v)^t)=\psi(v_{1,1}),\ v\in \symspace_{2\rkn}.
\end{gather*}

\subsection{Other notation} \label{sec: othernotation}
\begin{itemize}
\item We use the notation $a\ll_d b$ to mean that $a\leq cb$ with $c>0$ a constant depending on $d$.

\item For any $g\in G$ define $\wgtt{g}\in\R_{>0}$ by $\wgtt{u\levi(m)k}=\abs{\det m}$
for any $u\in U$, $m\in\GLnn$, $k\in K$. Let $\wgt{g}=\wgtt{\toG(g)}$ for $g\in G'$.

\item $\csgr(Q)$ is the set of compact open subgroups of a topological group $Q$.

\item $\csgr^s(G')$ is the subset of $\csgr(G')$ consisting of the $K_0$'s for which
$k\in K_0\mapsto \animg{k}$ is an isomorphism of groups. We identify any $K_0\in\csgr^s(G')$
with its image under $k\mapsto \animg{k}$ (an element of $\csgr(\tilde G)$).
Any sufficiently small $K_0\in\csgr(G')$ belongs to $\csgr^s(G')$.

\item For an $\ell$-group $Q$ let $C(Q)$ (resp., $\swz(Q)$) be the space of continuous (resp., Schwartz) functions on $Q$ respectively.

\item When $F$ is $p$-adic, if $Q'$ is a closed subgroup of $Q$ and $\chi$ is a character of $Q'$, we denote by $C(Q'\bs Q,\chi)$
(resp., $C^{\smth}(Q'\bs Q,\chi)$,  $\swrz(Q'\bs Q,\chi)$)  the spaces of
continuous (resp. $Q$-smooth\footnote{i.e., right-invariant under an open subgroup of $Q$},
smooth and compactly supported modulo $Q'$)
complex-valued left $(Q',\chi)$-equivariant functions on $Q$.

\item
For an $\ell$-group $Q$ we write $\Irr Q$ for the set of equivalence classes of irreducible representations of $Q$.
If $Q$ is reductive, we also write $\Irr_{\sqr}Q$ and $\Irr_{\temp}Q$ for the subsets of irreducible unitary square-integrable
(modulo center) and tempered representations respectively. We write $\Irr_{\gen}\GLnn$ and $\Irr_{\meta}\GLnn$
for the subset of irreducible generic representations of $\GLnn$ and representations of metaplectic type (see below), respectively.
For the set of irreducible generic representations of $\tilde G$ we use the notation $\Irr_{\gen,\psi_{\tilde N}}\tilde G$ to emphasize the dependence
on the character $\psi_{\tilde N}$.

\item For $\pi\in \Irr Q$, let $\d\pi$ be the contragredient of $\pi$.

\item For $\pi\in\Irr_{\gen}\GLnn$, $\WhitM(\pi)$ denotes the (uniquely
determined) Whittaker space of $\pi$ with respect to the character $\psi_{N_\GLnn}$. Similarly we use the notation
$\WhitMd$, $\WhitML$, $\WhitMLd$, $\WhitG$, $\WhitGd$.

\item For $\pi\in\Irr_{\gen}\Levi$ let $\Ind(\WhitML(\pi))$ be the space of smooth left $U$-invariant functions $W:G\rightarrow\C$ such that
for all $g\in G$, the function $m\mapsto\modulus_P(m)^{-\frac12}W(mg)$ on $\Levi$ belongs to $\WhitML(\pi)$. Similarly define $\Ind(\WhitMLd(\pi))$.

\item If a group $G_0$ acts on a vector space $W$ and $H_0$ is a subgroup of $G_0$,
we denote by $W^{H_0}$ the subspace of $H_0$-fixed points.

\item We use the following bracket notation for iterated integrals:
$\iint \,( \iint\,\ldots)\ldots$ implies that the inner integrals converge
as a double integral and after evaluating them,
the outer double integral is absolutely convergent.
\end{itemize}

\subsection{Measures}\label{sec: chevalley}
We take the self-dual Haar measure on $F$ with respect to $\psi$.
We use the following convention for Haar measures for algebraic subgroups of $G$.
(We consider $G'$, $\GLnn$, $\GLn$ as subgroups of $G$ through the embeddings $\toG$, $\levi$ and $\toG\circ\levi'$.)

When $F$ is a $p$-adic field, let $\OO$ be its ring of integers.
The Lie algebra $\mathfrak{M}$ of $\GL_{4\rkn}$ consists of the $4\rkn\times 4\rkn$-matrices $X$ over $F$.
Let $\mathfrak{M}_{\OO}$ be the lattice of integral matrices in $\mathfrak{M}$.
For any algebraic subgroup $\alg{Q}$ of $\GL_{4\rkn}$ defined over $F$
(e.g., an algebraic subgroup of $\alg{G}$)
let $\mathfrak{q}\subset\mathfrak{M}$ be the Lie algebra of $\alg{Q}$.
The lattice $\mathfrak{q}\cap \mathfrak{M}_{\OO}$ of $\mathfrak{q}$ gives rise to a gauge form of $\alg{Q}$
(determined up to multiplication by an element of $\OO^*$) and we use it to define a Haar measure on $Q$ by the recipe of \cite{MR0217077}.

When $F=\R$, the measures are fixed similarly, except that we use $\mathfrak{M}_{\Z}$ in place of $\mathfrak{M}_{\OO}$.
When $F=\C$ we will only consider subgroups of $G$ which are defined over $\R$ and take the gauge forms induced from the above
recipe for $\R$.

\subsection{Weil representation} \label{sec: Weil}
Let $\symp$ be a symplectic space over $F$ with a symplectic form $\sprod{\cdot}{\cdot}$.
Let $\Heise=\Heise_\symp$ be the Heisenberg group of $(\symp,\sprod{\cdot}{\cdot})$.
Recall that $\Heise_\symp=\symp\oplus F$ with the product rule
\[
(x,t)\cdot (y,z)=(x+y,t+z+\frac12\sprod xy).
\]
Fix a polarization $\symp=\symp_+\dsum \symp_-$. The group $\Sp(\symp)$ acts on the right on $\symp$.
We write a typical element of $\Sp(\symp)$ as
$\sm ABCD$ where $A\in \Hom(\symp_+,\symp_+)$, $B\in \Hom(\symp_+,\symp_-)$, $C\in \Hom(\symp_-,\symp_+)$ and $D\in \Hom (\symp_-,\symp_-)$.
Let $\Mp(\symp)$ be the metaplectic two-fold cover of $\Sp(\symp)$ with respect to the Rao cocycle determined
by the splitting.
Consider the Weil representation $\wev$ of the group $\Heise_\symp\rtimes\Mp(\symp)$ on $\swz(\symp_+)$.
Explicitly, for any $\Phi\in\swz(\symp_+)$ and $X\in \symp_+$ the action of $\Heise_\symp$ is given by
\begin{subequations}
\begin{align}
\label{eq: weilH1} \wev(a,0)\Phi(X)&=\Phi(X+a), \ \ a\in \symp_+,\\
\label{eq: weilH2} \wev(b,0)\Phi(X)&=\psi(\sprod Xb)\Phi(X),\ \ b\in \symp_-,\\
\label{eq: weilH3} \wev(0,t)\Phi(X)&=\psi(t)\Phi(X), \ \ t\in F,
\end{align}
\end{subequations}
while the action of $\Mp(\symp)$ is (partially) given by
\begin{subequations}
\begin{align}
\wev(\sm g00{g^*},\epsilon)\Phi(X)&=\epsilon\weilfctr(\det g)\abs{\det g}^{\frac12}\Phi(X g), \ \ g\in\GL(\symp_+),\\
\wev(\sm IB0I,\epsilon)\Phi(X)&=\epsilon\psi(\frac12\sprod{X}{XB})\Phi(X),\ \ B\in\Hom(\symp_+,\symp_-)\text{ self-dual},
\end{align}
\end{subequations}
where $\weilfctr$ is Weil's factor.

We now take $\symp=F^{2{\rkn}}$ with the standard symplectic form
\[
\sprod{(x_1,\dots,x_{2{\rkn}})}{(y_1,\dots,y_{2{\rkn}})}=\sum_{i=1}^{\rkn}x_iy_{2{\rkn}+1-i}-\sum_{i=1}^{\rkn}
y_ix_{2{\rkn}+1-i}
\]
and the standard polarization $\symp_+=\{(x_1,\dots,x_{\rkn},0,\dots,0)\}$, $\symp_-=\{(0,\dots,0,y_1,\dots,y_{\rkn})\}$.
(We identify $\symp_+$ and $\symp_-$ with $F^{\rkn}$.)
The corresponding Heisenberg group is isomorphic to the quotient $V/\altV$
(with $V$, $\altV$ as defined in \S\ref{sec: elements})
via $v\mapsto\Hei{v}:=((v_{{\rkn},{\rkn}+j})_{j=1,\dots,2{\rkn}},\frac12v_{{\rkn},3{\rkn}+1})$.

For $X=(x_1,\dots,x_{\rkn}),X'=(x_1',\dots,x_{\rkn}')\in F^{\rkn}$ define
\[
\sprod X{X'}'=x_1x'_{\rkn}+\dots+x_{\rkn}x'_1.
\]
For $\Phi\in\swz(F^{\rkn})$ define the Fourier transform
\[
\hat\Phi(X)=\int_{F^{\rkn}}\Phi(X')\psi(\sprod{X}{X'}')\ dX'.
\]
Then, realized on $\swz(F^{\rkn})$, the Weil representation satisfies
\begin{subequations}
\begin{align}
\label{eq: weil1} \wev(\animg{\levi'(h)})\Phi(X)&=\abs{\det(h)}^{\frac{1}{2}}
\beta_{\psi}(\levi'(h))\Phi(Xh),\ \ h\in\GLn,\\
\label{eq: weil2} \wev(\animg{\sm1B{}1})\Phi(X)&=\psi(\frac12\sprod{X}{XB}')\Phi(X),\ \ \ B\in\symspace_{\rkn},\\
\label{eq: weil3} \wev(\twU)\Phi(X)&=\beta_{\psi}(w'_{U'})\hat\Phi(X).
\end{align}
\end{subequations}
Here $\beta_\psi(g)$, $g\in G'$ are certain roots of unity; moreover $\beta_{\psi}(\levi'(h))=\weilfctr(\det h)$.

We extend $\weil_{\psi}$ to $V\rtimes\tilde G$ by setting
\begin{equation}\label{eq: weilext}
\wev(v \animg g)\Phi=\psi(v_{1,2}+\dots+v_{{\rkn}-1,{\rkn}})^{-1}\wev(\Hei{v})(\wev(\animg g)\Phi),\,\,v\in V,\,\,g\in G'.
\end{equation}
Then for any $g\in G'$, $v\in V$ we have
\begin{equation}\label{eq: weilext2}
\wev((\toG(g)v\toG(g)^{-1})\animg g)\Phi=\wev(\animg g)(\wev(v)\Phi).
\end{equation}

\subsection{Stable integral}
For the rest of the section, we assume $F$ is $p$-adic.

Suppose that $U_0$ is a unipotent group over $F$ with a fixed Haar measure $du$.
Recall that the group generated by a relatively compact subset of $U_0$ is relatively compact.
In particular, the set $\csgr(U_0)$ is directed. Recall the following definition of stable integral in \cite{MR3267120}.

\label{sec: stint}
\begin{definition} \label{def: stable integral}
Let $f$ be a smooth function on $U_0$.
We say that $f$ has a \emph{stable integral} over $U_0$ if there exists $U_1\in\csgr(U_0)$ such that for any
$U_2\in\csgr(U_0)$ containing $U_1$ we have
\begin{equation} \label{eq: comval}
\int_{U_2}f(u)\ du=\int_{U_1}f(u)\ du.
\end{equation}
In this case we write $\stint_{U_0} f(u)\ du$ for the common value \eqref{eq: comval} and say that
$\stint_{U_0} f(u)\ du$ stabilizes at $U_1$.
In other words, $\stint_{U_0} f(u)\ du$ is the limit of the net $(\int_{U_1}f(u)\ du)_{U_1\in\csgr(U_0)}$
with respect to the discrete topology of $\C$.
\end{definition}

Given a family of functions $f_x\in C^{\smth}(U_0)$, we say that the integral $\stint_{U_0}f_x(u)\ du$ \emph{stabilizes uniformly in $x$}
if $U_1$ as above can be chosen independently of $x$. Similarly, if $x$ ranges over a topological space $X$ then we say that
$\stint_{U_0}f_x(u)\ du$ stabilizes locally uniformly in $x$ if any $y\in X$ admits a neighborhood on which $\stint_{U_0}f_x(u)\ du$
stabilizes uniformly.

\subsection{A remark on convergence}
Frequently, we will make use of the following elementary remark.
\begin{remark} \label{rem: partint}
Let $\bf H$ be any algebraic group over $F$ and $\bf H'$ a closed subgroup.
Assume that $\modulus_H\rest_{H'}\equiv\modulus_{H'}$. Suppose that $f\in C^{\smth}(H)$ and that the integral
$\int_H f(h)\ dh$ converges absolutely. Then the same is true for $\int_{H'}f(h')\ dh'$.
\end{remark}

\section{Statement of main result}

\subsection{Local Fourier--Jacobi transform}
For any $f\in C(G)$ and $s\in\C$ define $f_s(g)=f(g)\wgtt{g}^s$, $g\in G$.
Let $\pi\in\Irr_{\gen}\Levi$
with Whittaker model $\WhitML(\pi)$.
Let $\Ind(\WhitML(\pi))$ be the space of $G$-smooth left $U$-invariant functions $W:G\rightarrow\C$ such that
for all $g\in G$, the function $\modulus_P(m)^{-\frac12}W(mg)$ on $\Levi$ belongs to $\WhitML(\pi)$.
For any $s\in\C$ we have a representation $\Ind(\WhitML(\pi),s)$ on the space $\Ind(\WhitML(\pi))$ given by
$(I(s,g)W)_s(x)=W_s(xg)$, $x,g\in G$. It is equivalent to the induced representation of $\pi\otimes\nu^s$ from $P$ to $G$.
The family $W_s$, $s\in \C$ is a holomorphic section of this family of induced representations.

Let $F$ be a $p$-adic field. Following \cite{MR1675971}, for any $W\in C^{\smth}(N\bs G,\psi_N)$ and $\Phi\in \swz(F^{\rkn})$ define a genuine function on $\tilde G$:
\begin{equation}\label{eq: defwhitform}
\whitform(W,\Phi,\animg g)=\int_{V_\gamma\bs V} W(\gamma v \toG(g))\wevinv(v\animg g)\Phi(\xi_{\rkn})\,dv,\,\,\,g\in G'
\end{equation}
where the element $\gamma\in G$ and the groups $V$ and $V_\gamma$ were defined in \S\ref{sec: elements}.

Its properties were studied in \cite[\S4]{1401.0198}.
In particular, the integrand is always compactly supported and $\whitform$ gives rise to a $V\rtimes\tilde G$-intertwining map
\begin{equation} \label{eq: whitformequ}
\whitform:C^{\smth}(N\bs G,\psi_N)\otimes\swz(F^{\rkn})\rightarrow C^{\smth}(N'\bs\tilde G,\psi_{\tilde N})
\end{equation}
where $V\rtimes\tilde G$ acts via $V\rtimes\toG(G')$ by right translation on $C^{\smth}(N\bs G,\psi_N)$, through $\wevinv$ on $\swz(F^{\rkn})$
and via the projection to $\tilde G$ by right translation on $C^{\smth}(N'\bs\tilde G,\psi_{\tilde N})$.

Let $\Vp\subset V$ be the image under $\toU_\Levi$ of the space of ${\rkn}\times {\rkn}$-matrices whose rows are zero except possibly for the last one.
For $c=\toU_\Levi(x)\in \Vp$ we denote by $\underline c\in F^{\rkn}$ the last row of $x$.
Then
\[
\whitform(W(\cdot c),\Phi(\cdot+\underline{c}),\animg{g})=\whitform(W,\Phi,\animg{g}),\ \ c\in \Vp, g\in G'.
\]
It follows that the function $\whitform(W,\Phi,\cdot)$ factors through $W\otimes\Phi\mapsto\Phi*W$
where for any function $f\in C^{\infty}(G)$ we set
\[
\Phi*f(g)=\int_{\Vp} f(g c)\Phi(\underline{c}) \,dc.
\]
We will denote by $\nwhitform$ the map
\[
\nwhitform:C^{\smth}(N\bs G,\psi_N)\rightarrow C^{\smth}(\tilde N\bs\tilde G,\psi_{\tilde N})
\]
such that
\[
\whitform(W,\Phi,\cdot)=\nwhitform(\Phi*W,\cdot).
\]
(Informally, $\nwhitform(W,\cdot)=\whitform(W,\delta_0,\cdot)$ where $\delta_0$ is the delta function at $0$.)
The map $\nwhitform$ is no longer $\tilde G$ or $V$-equivariant since neither $\tilde G$ nor $V$ (acting through $\wevinv$) stabilize $\delta_0$.
However, $\nwhitform$ satisfies the following equivariance property.
Let $\Vm=V_U\rtimes \altV_{\Levi}$ where $V_U=V\cap U$ and $\altV_{\Levi}\subset V_{\Levi}$ is defined in \S\ref{sec: elements}.
Thus, $\Vm$ is the preimage of $\symp_-$ under the composition $V\rightarrow V/\altV\simeq\Heise\rightarrow\symp$ and we have $V=\Vp\ltimes \Vm$.
Note that $\Vm$ is normalized by $P'$.
Let $\psi_{\Vm}$ be the character on $\Vm$ given by
\[
\psi_{\Vm}(vu)=\psi_{N_\Levi}(v)^{-1}\psi_U(u),\ \ v\in \altV_{\Levi}, u\in V_U.
\]

\begin{lemma} \label{lem: equnwhit}
For any $v\in \Vm$, $p=mu\in P'$ where $m\in M'$ and $u\in U'$, we have
\[
\nwhitform(W(\cdot v\toG(p)),\animg{g})=
\wgt{m}^{\frac12}\beta_{\psi^{-1}}(m)^{-1}\psi_{\Vm}(v)\nwhitform(W,\animg{g}\animg{p}).
\]
\end{lemma}

\begin{proof}
The statement simply reflects the fact that $\Vm\rtimes\tilde P$ acts on $\delta_0$ by multiplication by the character
\[
\animg{m'}\animg{u'}v\mapsto\wgt{m'}^{-\frac12}\beta_{\psi^{-1}}(m')\psi_{\Vm}^{-1}(v),\ \ \ v\in \Vm,m'\in M',u'\in U'.
\]
More rigorously, fix $W$, $p$ and $v$ and let $C$ be a small neighborhood of $0$ in $F^{\rkn}$.
Suppose that $\Phi\in\swrz(F^{\rkn})$ is supported in $C$ and $\int_{F^{\rkn}}\Phi(c)\ dc=1$.
Then $\whitform(W,\Phi,\animg{g})=\nwhitform(W,\animg{g})$ for all $g\in G'$ provided that $C$ is sufficiently small.
On the other hand, if $\Phi'=\wevinv(v\animg{p})\Phi$ then
\[
\whitform(W(\cdot v\toG(p)),\Phi',\animg{g})=\whitform(W,\Phi,\animg{g}\animg{p}).
\]
By \eqref{eq: weilH2}, \eqref{eq: weilH3}, \eqref{eq: weil1}, \eqref{eq: weil2} and \eqref{eq: weilext},
$\operatorname{supp}\Phi'\subset Cm^{-1}$ and (when $C$ is sufficiently small)
\[
\int_{F^{\rkn}}\Phi'(c)\ dc=\psi_{\Vm}(v)^{-1}\beta_{\psi^{-1}}(m)\wgt{m}^{-\frac12}.
\]
Thus, if $C$ is small enough then
\[
\Phi'*W(\cdot vp)=\psi_{\Vm}(v)^{-1}\beta_{\psi^{-1}}(m)\wgt{m}^{-\frac12}W(\cdot vp).
\]
Hence,
\[
\whitform(W(\cdot v\toG(p)),\Phi',\animg{g})=\psi_{\Vm}(v)^{-1}\beta_{\psi^{-1}}(m)\wgt{m}^{-\frac12}\nwhitform(W(\cdot v\toG(p)),\animg{g}).
\]
The lemma follows.
\end{proof}

For later reference we record the following result which is implicit in the proof of \cite[Lemma 4.5]{1401.0198}.

\begin{lemma}\label{lem: support subsym}
Suppose that $F$ is $p$-adic. Then for any $K_0\in\csgr(G)$ there exists $\Omega\in\csgr(V_U)$ such that
for any $W\in C(N\bs G,\psi_N)^{K_0}$ the support of $W(\gamma\cdot)\rest_{\Vm}$ is contained in $V_\gamma\toG(w'_{U'})\Omega\toG(w'_{U'})^{-1}$.
\end{lemma}

\begin{remark}
When $F$ is an archimedean field, the above discussion still holds for $W\in \Ind(\WhitML(\pi))$
(or more generally if $W\in C^{\smth}(N\bs G,\psi_N)$ is of moderate growth).
From \cite[Lemma~4.9]{1401.0198}, the integral \eqref{eq: defwhitform} converges.
Moreover $\whitform(W,\Phi,\cdot)$ factors through $W\otimes\Phi\mapsto\Phi*W$, and the induced linear form $\nwhitform$ satisfies the
equivariance property in Lemma~\ref{lem: equnwhit} (by an approximate identity argument).
\end{remark}
\subsection{Explicit local descent}

Let $F$ be a local field of characteristic 0. Define the intertwining operator
\[
M(\pi,s)=M(s):\Ind(\WhitML(\pi),s)\rightarrow\Ind(\WhitML(\d\pi),-s)
\]
by (the analytic continuation of)
\begin{equation} \label{eq: defM}
M(s)W(g)=\wgtt{g}^s\int_U W_s(\levi(\spclt) w_U ug)\,du
\end{equation}
where $\spclt=\diag(1,-1,\dots,1,-1)$ is introduced in order to preserve the character $\psi_{N_\Levi}$.
By abuse of notation we will also denote by $M(\pi,s)$ the intertwining operator
$\Ind(\WhitMLd(\pi),s)\rightarrow\Ind(\WhitMLd(\d\pi),-s)$ defined in the same way.

For simplicity we denote $M^*_sW:=(M(s)W)_{-s}$ so that
\[
M^*_sW=\int_U W_s(\levi(\spclt) w_U u\cdot)\ du
\]
for $\Re s\gg_{\pi}1$. Set $M^*W:=M^*_{\frac12}W$.

Recall that $H_\GLnn$ is the centralizer of $E=\diag(1,-1,\dots,1,-1)(=\spclt)$, isomorphic to $\GL_{\rkn}\times\GL_{\rkn}$.
The involution $\wnn$ lies in the normalizer of $H_\GLnn$.
We consider the class $\Irr_{\meta}\GLnn$ of irreducible representations
of $\GLnn$ which admit a continuous non-zero
$H_\GLnn$-invariant linear form $\ell$ on the space of $\pi$.
It is known that any such $\pi$ is self-dual and $\ell$ is unique up to a scalar (\cite{MR1394521, MR2553879}).
Thus, $\ell\circ\pi(\wnn)=\eps_\pi\ell$ where $\eps_\pi\in\{\pm1\}$ does not depend on the choice of $\ell$.
By \cite[Theorem~3.2]{1401.0198}, when $F$ is $p$-adic, for any $\pi\in\Irr_{\meta,\gen}\GLnn$ we have
\begin{equation} \label{eq:shalmetagen}
\eps_\pi=\epsilon(\frac12,\pi,\psi)
\end{equation}
where $\epsilon(s,\pi,\psi)$ is the standard $\epsilon$-factor attached to $\pi$.

Let $\pi\in\Irr_{\gen,\meta}\GLnn$, considered also as a representation of $\Levi$ via $\levi$.
By \cite[Proposition~4.1]{1401.0198} $M(s)$ is holomorphic at $s=\frac12$.
Denote by $\des(\pi)$ the space of Whittaker functions on $\tilde G$ generated by $\whitform(M^*W,\Phi,\cdot)$,
$W\in \Ind(\WhitML(\pi))$, $\Phi\in\swz(F^{\rkn})$, i.e.,
\[
\des(\pi)=\{\nwhitform(M^*W,\cdot):W\in \Ind(\WhitML(\pi))\}.
\]
This defines an explicit descent map $\pi\mapsto \des(\pi)$ on $\Irr_{\gen,\meta}\GLnn$.
By \cite[Theorem in \S1.3]{MR1671452} $\des(\pi)\ne0$. It can be shown that $\des(\pi)$ is admissible but we will not do it here since
we will not use this fact directly.

\subsection{\emph{Good} representations}

Let $\pi\in\Irr_{\gen}\Levi$ and $\tilde\sigma\in\Irr_{\gen,\psi_{\tilde N}^{-1}}\tilde G$ with Whittaker model $\WhitGd(\tilde\sigma)$.
Following Ginzburg--Rallis--Soudry \cite{MR1675971}, for any $\tilde W\in \WhitGd(\tilde\sigma)$, $W\in\Ind(\WhitML(\pi))$ define
the local Shimura type integral
\begin{equation}\label{eq: localinner}
\tilde{J}(\tilde W,W,s):=\int_{N'\bs G'}\tilde W(\tilde g)\nwhitform(W_s,\tilde g)\ dg.
\end{equation}
By \cite[\S6.3]{MR1675971} and \cite{MR1671452} $\tilde{J}$ converges for $\Re s\gg_{\pi,\tilde\sigma}1$ and admits a meromorphic continuation in $s$.
Moreover, for any $s\in\C$ we can choose $\tilde W$ and $W$ such that $\tilde{J}(\tilde W,W,s)\ne0$.

Let $\pi\in\Irr_{\gen,\meta}\GLnn$. We say that $\pi$ is \emph{\good}\ if the following conditions are satisfied for all $\psi$:
\begin{enumerate}
\item $\des(\pi)$ is irreducible.
\item $\tilde{J}(\tilde W,W,s)$ is holomorphic at $s=\frac12$ for any $\tilde W\in\desinv(\pi)$ and $W\in\Ind(\WhitML(\pi))$.
(We do not assume that the integral defining $\tilde{J}$ converges at $s=\frac12$.)
\item There is a non-degenerate $\tilde G$-invariant pairing $[\cdot,\cdot]$ on $\desinv(\pi)\times\des(\pi)$ such that
\[
\tilde{J}(\tilde W,W,\frac12)=[\tilde W,\nwhitform(M^*W,\cdot)]
\]
for any $\tilde W\in\desinv(\pi)$, $W\in\Ind(\WhitML(\pi))$.
In particular $\des(\pi)^\vee\simeq\desinv(\pi)$.
\end{enumerate}

This property was introduced and discussed in \cite[\S5]{1401.0198}.
In particular, if $\pi$ is \good\ and $\tilde\pi=\desinv(\pi)$ then there exists a constant $c_\pi$ such that for any
$\tilde W\in\WhitGd(\tilde\pi^\vee)$, $\d{\tilde W}\in \WhitG(\tilde\pi)$ we have
\begin{equation} \label{eq: main}
\stint_{N'}[\tilde\pi(\animg n)\tilde W, \d{\tilde W}]\psi_{\tilde N}( n)\,dn=c_\pi\tilde W(e)\d{\tilde W}(e).
\end{equation}
More explicitly, for any $W\in\Ind(\WhitML(\pi))$ and $\alt{W}\in\Ind(\WhitMLd(\pi))$ we have
\begin{equation} \label{eq: main1}
\stint_{N'}\tilde{J}(\nwhitformd(M^*\alt{W},\cdot \animg n),W,\frac12)\psi_{\tilde N}( n)\,dn
=c_\pi\nwhitformd(M^*\alt{W},e)\nwhitform(M^*W,e).
\end{equation}

Let $F$ be a $p$-adic field. In the rest of the paper we will prove the following statement:
\begin{theorem} \label{thm: main}
For any unitarizable $\pi\in\Irr_{\gen,\meta}\GLnn$ which is \good\ we have $c_\pi=\eps_\pi$.
\end{theorem}
\begin{remark}
Of course, we expect Theorem \ref{thm: main} to hold in the archimedean case as well.
However, we will not deal with the archimedean case in this paper.
\end{remark}

Of particular importance will be the following special case (see Remark \ref{rem: sqrintcase} below):
\begin{corollary} \label{cor: sqrintcase}
Suppose $\pi=\tau_1\times\dots\times\tau_l$ where $\tau_j\in\Irr_{\sqr,\meta}\GL_{2m_j}$,
$j=1,\dots,l$ are distinct and ${\rkn}=m_1+\dots+m_l$. Then $\pi$ is \good\ and $\tilde\pi:=\desinv(\pi)$ is square-integrable.
Moreover,
\[
\int_{N'}\tilde{J}(\tilde\pi(\animg n)\tilde W,W,\frac12)\psi_{\tilde N}( n)\,dn
=\eps_\pi\tilde W(e)\nwhitform(M^*W,e),
\]
for any $W\in\Ind(\WhitML(\pi))$ and $\tilde W\in\WhitGd(\tilde\pi)$.
\end{corollary}

\subsection{Relation to global statement}

Suppose now that $F$ is a number field and $\A$ is its ring of adeles.
We say that an irreducible cuspidal representation $\pi$ of $\GLnn$ is of \emph{metaplectic type} if
\[
\int_{H_\GLnn(F)\bs H_\GLnn(\A)\cap\GLnn(\A)^1}\varphi(h)\ dh\ne0
\]
for some $\varphi$ in the space of $\pi$. (Here, $\GLnn(\A)^1=\{m\in\GLnn(\A):\abs{\det m}=1\}$.)
Equivalently, $L^S(\frac12,\pi)\res_{s=1}L^S(s,\pi,\wedge^2)\ne0$ (\cite{MR1241129}).\footnote{We can
replace the partial $L$-function by the completed one since the local factors are holomorphic and non-zero.}
In particular, $\pi$ is self-dual and admits a trivial central character.
We write $\Cusp_{\meta}\GLnn$ for the set of irreducible cuspidal representations of metaplectic type.

Consider the set $\Mcusp\GLnn$ of automorphic representations $\pi$ of $\GLnn(\A)$ which are realized on Eisenstein series
induced from $\pi_1\otimes\dots\otimes\pi_k$ where $\pi_i\in\Cusp_{\meta}\GL_{2n_i}$, $i=1,\dots,k$ are distinct and $\rkn=n_1+\dots+n_k$.
The representation $\pi$ is irreducible: it is equivalent to the parabolic induction $\pi_1\times\dots\times\pi_k$.
Moreover, $\pi$ determines $\pi_1,\dots,\pi_k$ uniquely up to permutation \cite{MR618323, MR623137}.

Recall that Conjecture~\ref{conj: metplectic global} is pertaining to the representations $\pi\in\Mcusp\GLnn$.
The following fact will be crucial for us.

\begin{proposition}(\cite[Theorem~6.2]{1401.0198}) \label{prop: globalgood}
If $\pi\in \Mcusp\GLnn$ then its local components $\pi_v$ are \good.
\end{proposition}

Thus, Theorem~\ref{thm: main} implies Theorem~\ref{thm: intromain}, our main result.

\subsection{Relation to Bessel functions}
We record here a \emph{purely formal} argument that relates Theorem~\ref{thm: main} to an identity of Bessel functions, which are defined in \cite{MR3021791}.
We do not worry about convergence issues in this subsection.

Using  the functional equation (cf.~\cite{MR3267119}, note that the central character of $\pi$ is trivial)
\begin{equation} \label{eq: local functional equation}
\tilde{J}(\tilde W,M(s)W,-s)=\abs{2}^{2\rkn s} \frac{\gamma(\tilde\pi\otimes\pi,s+\frac12,\psi)}{\gamma(\pi,s,\psi)\gamma(\pi,\wedge^2,2s,\psi)}
\tilde{J}(\tilde W,W,s),
\end{equation}
\eqref{eq: main} becomes:
\[
\int_{N'}\tilde{J}(\tilde\pi(\animg n)\tilde W,M(\frac12)W,-\frac12)\psi_{\tilde N}( n)\,dn
=\abs{2}^{\rkn}\frac{\gamma(\tilde\pi\otimes\pi,s+\frac12,\psi)}{\gamma(\pi,s,\psi)\gamma(\pi,\wedge^2,2s,\psi)}\rest_{s=\frac12}
c_\pi\tilde W(e)\nwhitform(M^*W,e).
\]
Explicitly, the left-hand side is
\[
\int_{N'}\int_{N'\bs G'}\tilde W(\tilde g\tilde n)\nwhitform(M^*W,\tilde g)\ dg\psi_{\tilde N}( n)\,dn
=\int_{N'}\int_{N'\bs G'}\tilde W(\tilde g\tilde n)\tilde W'(\tilde g)\ dg\psi_{\tilde N}( n)\,dn
\]
where $\tilde W'=\nwhitform(M^*W,\cdot)$. Using Bruhat decomposition this is
\begin{multline*}
\int_{T'}\int_{N'}\int_{N'}\delta_{B'}(t)\tilde W(\tilde w_0'\tilde t\tilde n'\tilde n)
\tilde W'(\tilde w_0'\tilde t\tilde n')\psi_{\tilde N}( n)\ dn'\ dn\ dt
\\=\int_{T'}\int_{N'}\int_{N'}\delta_{B'}(t)\tilde W(\tilde w_0'\tilde t\tilde n)
\tilde W'(\tilde w_0'\tilde t\tilde n')\psi_{\tilde N}( n)\psi_{\tilde N}^{-1}( n')\ dn'\ dn\ dt.
\end{multline*}
By definition of Bessel functions $\Bes_{\tilde\pi}^{\psi_{\tilde N}^{-1}}$ (see \eqref{eq: Bess rltn}), this is
\[
\tilde W(e)\tilde W'(e)\int_{T'}\Bes_{\tilde\pi}^{\psi_{\tilde N}^{-1}}(\tilde w_0'\tilde t)
\Bes_{\d{\tilde\pi}}^{\psi_{\tilde N}}(\tilde w_0'\tilde t)\delta_{B'}(t)\ dt.
\]
Here $\d{\tilde\pi}= \des(\pi)$. Thus, on a formal level, Theorem \ref{thm: main} becomes the following inner product identity:
\begin{equation}\label{eq: Besselinner}
\int_{T'}\Bes_{\tilde\pi}^{\psi_{\tilde N}^{-1}}(\tilde w_0'\tilde t)
\Bes_{\d{\tilde\pi}}^{\psi_{\tilde N}}(\tilde w_0'\tilde t)\delta_{B'}(t)\ dt=
\abs{2}^{\rkn}\frac{\gamma(\tilde\pi\otimes\pi,s+\frac12,\psi)}{\gamma(\pi,\wedge^2,2s,\psi)}\rest_{s=\frac12}.
\end{equation}
(It is easy to determine the order of the poles in the numerator and denominator on the right-hand side in terms
of the data of Theorem \ref{thm: genmetaclassification} below.)
Of course, the above manipulations are purely formal as the integrals are not absolutely convergent.

\subsection{A reduction of the main theorem}

\label{sec: firstred}

For the rest of the paper let $F$ be a $p$-adic field.
The proof of Theorem \ref{thm: main} is essentially local under the additional assumption
that $\pi\in\Irr_{\temp}\GLnn$ and $\tilde\pi:=\desinv(\pi)\in\Irr_{\temp}\tilde G$.
We first show that it indeed suffices to prove Theorem \ref{thm: main} under these additional assumptions.
This will use a global argument as well as the following classification result due to Matringe.
We denote by $\times$ parabolic induction for $\GL_m$.

\begin{theorem}[\cite{MR3430877}] \label{thm: genmetaclassification}
The set $\Irr_{\gen,\meta}\GLnn$ consists of the irreducible representations of the form
\[
\pi=\sigma_1\times\d\sigma_1\times\dots\times\sigma_k\times\d\sigma_k\times
\tau_1\times\dots\times\tau_l
\]
where $\sigma_1,\dots,\sigma_k$ are essentially square-integrable\footnote{that is, a twist of a square-integrable representation by
a quasi-character},
$\tau_1,\dots,\tau_l$ are square-integrable and $L(0,\tau_i,\wedge^2)=\infty$ for all $i$.
\end{theorem}

(We expect the same result to hold in the archimedean case as well.)

Denote by $\omega_\sigma$ the central character of $\sigma\in\Irr\GL_m$.
We also write $\sigma[s]$ for the twist of $\sigma$ by $\abs{\det\cdot}^s$.

Let $\tau_j\in\Irr_{\sqr,\meta}\GL_{2m_j}$, $j=1,\dots,l$ be distinct and $\delta_i\in\Irr_{\sqr}\GL_{n_i}$, $i=1,\dots,k$
with ${\rkn}=n_1+\dots+n_k+m_1+\dots+m_l$. (Possibly $k=0$ or $l=0$.) For $\underline{s}=(s_1,\dots,s_k)\in\C^k$ we consider the representation
\[
\pi(\underline{s})=\delta_1[s_1]\times\d{\delta_1}[-s_1]\times\dots\times\delta_k[s_k]\times\d{\delta_k}[-s_k]\times\tau_1\times\dots\times\tau_l.
\]
Suppose that our given $p$-adic field is the completion at a place $v$ of a number field.
We claim that for a dense set of $\underline{s}\in\iii\R^k$, $\pi(\underline{s})$ is the local component at $v$ of an element of $\Mcusp\GLnn$.
This follows from \cite[Appendix A]{1404.2909}.\footnote{We remark that the appendices and \S4 of \cite{1404.2909},
and in particular the proof of \cite[Theorem~3.1]{1404.2909}, are independent of the results of the current paper}
Indeed, let $m=m_1+\dots+m_l$ and let $\rho\in\Irr_{\sqr,\gen}\SO(2m+1)$ be the representation corresponding to
$\tau_1\times\dots\times\tau_l$ under Jiang--Soudry \cite{MR2058617} and let
$\tilde\rho\in\Irr_{\sqr,\gen}\Mp_m$ be the theta lift of $\rho$.
Let $\sigma(\underline{s})=\delta_1[s_1]\times\dots\times\delta_k[s_k]\rtimes\rho$ (parabolic induction) and
$\tilde\sigma(\underline{s})=\delta_1[s_1]\times\dots\times\delta_k[s_k]\rtimes\tilde\rho$.
Then for $\underline{s}$ in a dense open subset of $\iii\R^k$ we have $\sigma(\underline{s})\in\Irr\SO(2\rkn +1)$
and $\tilde\sigma(\underline{s})\in\Irr\tilde G$ and moreover by \cite{MR2999299} $\tilde\sigma(\underline{s})$
is the theta lift of $\sigma(\underline{s})$.
By \cite[Corollary A.8]{1404.2909}, for a dense set of $\underline{s}\in\iii\R^k$,
$\tilde\sigma(\underline{s})$ is the local component at $v$ of a generic cuspidal automorphic representation of $\Mp_{\rkn}(\A)$
whose theta lift to $\SO(2{\rkn}+1)$ is cuspidal.
The Cogdell--Kim--Piatetski-Shapiro--Shahidi lift (\cite{MR2075885}) of the latter to $\GLnn$ is the required representation in $\Mcusp\GLnn$.

It follows from Proposition~\ref{prop: globalgood}
that for a dense set of $\underline{s}\in\iii\R^k$, $\pi(\underline{s})\in\Irr_{\temp,\meta}\GLnn$ is \good.
Moreover, by \cite[Proposition 4.6]{1404.2909} $\desinv(\pi(\underline{s}))=\tilde\sigma(\underline{s})$ and in particular
it is tempered.

Suppose that $\underline{s}$ is in the domain
\[
\mathfrak{D}=\{(s_1,\dots,s_k)\in\C^k:-\frac12<\Re s_i<\frac12\text{ for all }i\}.
\]
We recall that by \cite[Lemma 4.12]{1401.0198} the integral defining $\tilde{J}(\nwhitformd(M^*\alt{W},\cdot),W,s)$
converges and is holomorphic at $s=\frac12$ for any $W\in\Ind(\WhitML(\pi(\underline{s})))$, $\alt{W}\in\Ind(\WhitMLd(\pi(\underline{s})))$.
Moreover, by the properties of $\whitform$ and $\tilde J$ (see \cite{1401.0198}, (4.4) and (4.13)), the function
$g\mapsto\tilde{J}(\nwhitformd(M^*\alt{W},\cdot \animg g),W,\frac12)$ is bi-$K_0$-invariant
provided that $K_0\in\csgr^s(G')$ is such that $I(\frac12,\toG(k))W=W$, $I(\frac12,\toG(k))\alt{W}=\alt{W}$ for all $k\in K_0$.
It follows from \cite[Proposition 2.11]{MR3267120} (applied to $\tilde G$) that the stable integral on the left-hand side of \eqref{eq: main1} can be written
as an integral over a compact open subgroup of $N'$ depending only on $K_0$.
Thus, taking $W\in\Ind(\WhitML(\pi(\underline{s})))$, $\alt{W}\in\Ind(\WhitMLd(\pi(\underline{s})))$
to be Jacquet integrals, both sides of \eqref{eq: main1} for $\pi(\underline{s})$ are holomorphic functions of $\underline{s}\in\mathfrak{D}$.
Note that by \cite[Lemma~3.6]{1401.0198} $\eps_{\pi(\underline{s})}=\omega_{\delta_1}(-1)\cdots\omega_{\delta_k}(-1)\eps_{\tau_1}\cdots\eps_{\tau_l}$
is independent of $\underline{s}$.
Thus, in order to prove \eqref{eq: main1} for $\underline{s}\in\mathfrak{D}$ it is enough to show it for a dense set of $\underline{s}\in\iii\R^k$.
Since every unitarizable $\pi\in\Irr_{\meta}\GLnn$ is of the form $\pi(\underline{s})$ for some
$\tau_1,\dots,\tau_l,\delta_1,\dots,\delta_k$ as above and $\underline{s}\in\mathfrak{D}$, the reduction step follows.

\begin{remark} \label{rem: sqrintcase}
In the case $k=0$ we can globalize $\pi$ itself. Thus, by Proposition~\ref{prop: globalgood} $\pi$ is \good;
$\desinv(\pi)$ is irreducible and square integrable (see \cite[Theorem~3.1]{1404.2909}).\footnotemark[\value{footnote}]
This yields Corollary \ref{cor: sqrintcase} from Theorem \ref{thm: main}.
\end{remark}

In the remainder of the paper we will prove Theorem \ref{thm: main}, i.e., the Main Identity
\begin{multline} \tag{MI} \label{eq: MI'}
\stint_{N'}\big(\int_{N'\bs G'}\nwhitform(W_s,\tilde g)\nwhitformd(M^*\alt{W} ,\tilde g\tilde u)\ dg\big)
\psi_{\tilde N}(\tilde u)\ du\rest_{s=\frac12}=\\
\eps_\pi\nwhitformd(M^*\alt{W},e)\nwhitform(M^*W,e)
\end{multline}
under the assumptions that $\pi\in\Irr_{\temp,\meta}\GLnn$ is \good\ and $\tilde\pi:=\desinv(\pi)$ is tempered.

\begin{remark}
As already pointed out before (and used in the reduction step), by  \cite[Lemma 4.12]{1401.0198} the inner integral on the left-hand side of \eqref{eq: MI'}
converges (and is analytic) for $\Re s\ge\frac12$ (for any unitary $\pi$). We will not need to use this fact explicitly any further.
\end{remark}

\section{An informal sketch} \label{sec: nsketch}

Before embarking on the rigorous proof of the relation \eqref{eq: MI'}, let us give a brief sketch of the argument.
For convenience we will momentarily ignore all convergence issues (which are of course at the heart of the argument,
and will be discussed below). Thus, the discussion of this section is \emph{purely formal}.

Define (assuming convergent)
\[
\Bil(W,\d{W})=\int_{N'}\big(\int_{N'\bs G'}\nwhitform(W,\animg{g})
\nwhitformd(\d{W},\animg{g}\animg{u})\ dg\big)\psi_{\tilde N}(\animg{u})\ du,
\]
so that the left-hand side of \eqref{eq: MI'} is $\Bil(W_{\frac12},M^*W)$.

We will use the following formal identity (which follows from the Bruhat decomposition for $G'$):
for (suitable) $\tilde W\in C^{\smth}(\tilde N\bs \tilde G,\psi_{\tilde N})$ and $\d{\tilde W}\in C^{\smth}(\tilde N\bs \tilde G,\psi_{\tilde N}^{-1})$ we have
\begin{multline} \label{eq: BruhatG'}
\int_{N'}\big(\int_{N'\bs G'}\tilde W(\animg{g})\d{\tilde W}(\animg{g}\animg{u})\ dg\big)\psi_{\tilde N}(\animg{u})\ du=\\
\int_{T'}\modulus_{B'}(t)\int_{N'} \tilde W(\animg{w_0't}\animg{n_1})\psi_{\tilde N}(\animg{n_1})^{-1}\ dn_1
\int_{N'}\d{\tilde W}(\animg{w_0't}\animg{n_2})\psi_{\tilde N}(\animg{n_2})\ dn_2\ \ dt.
\end{multline}
(We remark that for the ensuing discussion it would be more natural to replace the right-hand side by
\begin{multline} \label{eq: lessadv}
\int_{N'_{\Levi'}}\int_{N'_{\Levi'}\bs \Levi'}  \modulus_{P'}(m)
\int_{U'}\tilde W(\twU\animg{m}\animg{n}\animg{u_1})\psi_{\tilde N}(\animg{u_1})^{-1}\ du_1\\
\int_{U'}\d{\tilde W}(\twU\animg{m}\animg{u_2})\psi_{\tilde N}(\animg{u_2})\ du_2\ \ dm\ \psi_{N'_{\Levi'}}(n)\ dn.
\end{multline}
However, it is analytically advantageous to work with the former expression.)

Applying this to $\tilde W=\nwhitform(W,\cdot)$, $\d{\tilde W}=\nwhitformd(\d{W},\cdot)$ we get
\begin{equation} \label{eq: bilYY}
\Bil(W ,\d{W},s):=\int_{T'}\oldhalf(W,t)\oldhalfd(\d{W},t)\modulus_{B'}({t})\,dt
\end{equation}
where
\[
\oldhalf(W,t):=\int_{N'}\nwhitform(W,\animg{w_0'tn})\psi_{\tilde N}(\animg n)^{-1}\,dn.
\]
We note the following equivariance property of $\oldhalf(W,t)$.
Let $\kgrp=\Vm\rtimes\toG(N')$, which is the stabilizer of $\psi_U$ in $N$.
Let $\psi_{\kgrp}$ be the character
\[
\psi_{\kgrp}(\toG(n)v)=\psi_{N'}(n)\psi_{\Vm}(v),\ \ n\in N',v\in\Vm
\]
on $\kgrp$. Then by Lemma~\ref{lem: equnwhit} $\oldhalf(W,t)$ is $(\kgrp,\psi_{\kgrp})$-equivariant in $W$.

Next, we explicate $\nwhitform(W,\animg{w_0'tn})$. By Lemma \ref{lem: equnwhit} it is enough to explicate $\nwhitform(W,\twU)$.
We have (up to an immaterial root of unity, which we suppress)
\[
\nwhitform(W,\twU)=*
\int_{\altV_{\Levi}\bs \Vm}W (w_Uv)\psi_{\Vm}(v)^{-1}\,dv=
*\int_{V_U}W (w_Uv)\psi_U(v)^{-1}\,dv.
\]
(See Remark \ref{rem: genbc} for a more precise statement.)
We also mention (this will be used for the right-hand side of \eqref{eq: MI'}) that
\[
\nwhitform(W,e)=\int_{\auxthree^{-1}(V_\gamma\rtimes\toG(N'))\auxthree\bs \kgrp} W(\gamma \auxthree n )\psi_{\kgrp}( n )^{-1}\, dn
\]
for a suitable unipotent element $\auxthree$. (See Lemma \ref{L: Mint}.) This expression is clearly $(\kgrp,\psi_{\kgrp})$-equivariant in $W$.

\begin{remark}
By \cite[\S 4.4, Theorem]{MR1671452} and \cite[Lemma 6.1]{MR3431601}, for $\pi\in \Irr_{\meta,\temp}\GLnn$ the space of
$(\kgrp,\psi_{\kgrp})$-equivariant linear forms on the Langlands quotient
of $\Ind(\WhitM(\pi),\frac12)$ is one-dimensional. This implies that $\oldhalf(M^*W,t)$ is proportional $\nwhitform(M^*W,e)$.
The constant of proportionality is a complicated function of $t$ which is probably related to the Bessel function.
At best, the putative equality becomes something like \eqref{eq: Besselinner} which seems difficult to approach directly.
Instead, we take a different approach.
\end{remark}

It follows from the formula for $\nwhitform(W,\twU)$ above that
\begin{equation} \label{eq: expforY}
\oldhalf(W,t)=*\wgt{t}^{{\rkn}-\frac12}\int_{N'_{\Levi'}}\int_U W(w_U\toG( \wnM tn)v)\psi_{N'_{\Levi'}}(n)^{-1}\psi_U(u)^{-1}\,du\ dn.
\end{equation}
Substituting this in \eqref{eq: bilYY} and using the Bruhat decomposition for $\GLn$ we obtain
\begin{multline} \label{eq: bFE}
\Bil(W ,\d{W})=\int_U\int_U\int_{N'_{\GLn}}\int_{N'_{\GLn}\bs\GLn}
W\left(\toLevi(gn)w_Uu_1\right)\d{W}\left(\toLevi(g)w_Uu_2\right)\\\modulus_P(\toLevi(g))^{-1}\abs{\det g}^{1-{\rkn}}\ dg\
\psi_{N'_{\GLn}}(n)\ dn\ \psi_U(u_1)^{-1}\psi_U(u_2)\,du_1\,du_2.
\end{multline}
For the inner integral, we will use the following functional equation proved in \cite{MR3220931}:
for any $W\in\WhitM(\pi)$ and $\d{W}\in\WhitMd(\d\pi)$ we have
\begin{multline*}
\int_{N'_{\GLn}\bs \GLn} W\left(\toM(g)\right)\d{W}\left(\toM(g)\right) \abs{\det g}^{1-{\rkn}}\,dg
\\=\int_{\zerocol}\int_{\zerocol}\int_{N'_{\GLn}\bs \GLn}
W \left(w_{2{\rkn},{\rkn}}\toM(g)\toU_{\GLnn}(X)\right)
\d{W} \left(w_{2{\rkn},{\rkn}}\toM(g)\toU_{\GLnn}(Y)\right)
\abs{\det g}^{{\rkn}-1}\,dg\,d X\,d Y
\end{multline*}
where $\zerocol$ is the subspace of $\Mat_\rkn$ consisting of the matrices whose first column is zero.
From this it is easy to infer that
\begin{multline} \label{eq: GLnnFEint}
\int_{N'_{\GLn}}\big(\int_{N'_{\GLn}\bs \GLn} W\left(\toM(gn)\right)\d{W}\left(\toM(g)\right) \abs{\det g}^{1-{\rkn}}\,dg\big)\
\psi_{N'_{\GLn}}(n)\,dn\\=
\int_{\vrq}\int_{\vrq}\int_{T'_{\GLn}}W\left(\toMd(t)\auxone r_1\right)
\d{W} \left(\toMd(t)\auxone r_2\right)\\
\abs{\det t}^{{\rkn}-1}\modulus_{B'_{\GLn}}(t)^{-1}\psi_{\vf}(r_1^{-1}r_2) \,dt\,dr_2\,dr_1
\end{multline}
where  $\vf=(\kgrp_\GLnn)^*$, $\psi_{\vf}$ is the character on $\vf$ given by $\psi_{\vf}(m)=\psi_{\kgrp_\GLnn}(m^*)$
and $\auxone\in\GLnn$ is a suitable element introduced in \S\ref{sec: bepsilon1}.
Substituting this in \eqref{eq: bFE} we obtain
\[
\Bil(W,\d{W})=\int_{\toMd(T'_{\GLn})}\newhalf(W,t)\newhalfd(\d{W},t)\,\frac{dt}{\abs{\det t}}
\]
where
\[
\newhalf(W,t):=\modulus_B^{-\frac12}(\levi(t))\int_{\bigvr}W\left(\levi(t\auxtwo\auxone) w_Uv\right)\psi_{\kgrp}(v)^{-1}\,dv,\ \ t\in \toMd(T'_{\GLn}).
\]
Here,
$\auxtwo\in\GLnn$ is an auxiliary unipotent element chosen so that $\levi(\auxtwo\auxone)w_U$ conjugates $\psi_U$
to a character which is supported on a single (short negative) root.
The sought-after identity \eqref{eq: MI'} is now a consequence of the following two results concerning $(\kgrp,\psi_{\kgrp})$-equivariant linear forms in $W$,
which are proved in \cite{MR3431601}.
\begin{enumerate}
\item $\newhalf(M^*W,t)$ is the constant function $\eps_\pi^{\rkn} \Mint(M^*W)$.
\item $\int_{\toMd(T'_{\GLn})}\newhalf(W_{\frac12},t)\,\frac{dt}{\abs{\det t}}=\eps_\pi^{{\rkn}+1}\Mint(M^* {W})$.
\end{enumerate}
It is here that we use that $\pi\in\Irr_{\meta}\GLnn$ in an essential way (and where the factor $\eps_\pi$ shows up).

This conclude the sketch of the argument. We finish with a few technical remarks about the rigorous justification of the above argument
and also give a few pointers to the upcoming sections.
Hopefully, this well shed some light on (if not motivate) the contents of the remaining sections.
\begin{enumerate}
\item To justify \eqref{eq: BruhatG'}, with $\int_{N'}$ replaced by $\stint_{N'}$ on both sides,
we are forced to assume that $\tilde W(\animg{w_0'tn})$ is compactly supported on $T'\times N'$
(see Lemma \ref{lem: bruhdec}).
This entails imposing a certain support condition on the section $W$ (see Definition \ref{def: spclW} and Lemma \ref{L: Iexp2}).
Fortunately (and this is a key point in our argument), it is enough to prove \eqref{eq: MI'} for a single pair $(W,\alt{W})$
for which the right-hand side is non-zero.
For the special sections as above, the non-vanishing is guaranteed by that
of the Bessel function of $\tilde\pi$, which is proved in the appendix, following ideas of Ichino--Zhang \cite{IZ}.
\item The fact that $\oldhalf$ is well-defined (with $\stint_{N'}$ instead of $\int_{N'}$) relies on results and techniques of Baruch
\cite{MR2192818} (modified to the case at hand in \cite{MR3021791}) on the stability of Bessel functions (Corollary \ref{cor: stbessel}).
This is the reason why we use the Bruhat decomposition with respect to $B'$ (i.e., \eqref{eq: BruhatG'}) rather than $P'$
(i.e., \eqref{eq: lessadv}).
\item Since we cannot control the support of $M^*\alt{W}$, it is necessary to justify \eqref{eq: expforY} for general sections.
To that end, we introduce a complex parameter $s$ and consider the family
\[
\Bil(W,\alt{W},s)=\Bil(W_s,\alt{W}_{-s}).
\]
The equality \eqref{eq: expforY} is then justified for $W_s$ with $\Re s\gg1$. (See Lemma \ref{L: Iexp}.)
This yields the analogue of \eqref{eq: bFE} for $\Bil(W,\alt{W},s)$ provided that $-\Re s\gg1$ and $W$ is special (Lemma \ref{L: defnewBil}).
\item A similar issue arises with the equation \eqref{eq: GLnnFEint}: we can only justify it if at least one of the Whittaker functions
satisfies an additional support constraint (which is in some sense dual to the support condition considered in the first remark).
(See Definition \ref{def: spclWM2} and Proposition \ref{prop: feuse}.)
Thus, we take $W$ to be special of the first kind and $\alt{W}$ to be special of the second kind (see Definition \ref{def: spclW2}).
\item Ultimately, we need to consider $\Bil(W,M(s)\alt{W},s)$ at $s=\frac12$.
In order to exploit the assumptions on $W$ and $\alt{W}$, we need to know that $M(s)$ is self-adjoint with respect to the bilinear form $\Bil$.
Such a relation was proved in \cite[Appendix B]{MR3220931} for a variant of $\Bil$ (where the order of integration is slightly different).
(See Corollary \ref{cor: FEBil}.)
For $-\Re s\gg1$ and $W$ special of the first kind, this agrees with the original $\Bil$.
\item The upshot is an identity
\[
\Bil(W,M(s)\alt{W},s)=\int_{\toMd(T'_{\GLn})}\newhalf(M^*_sW,t)\newhalfd(\alt{W}_s,t)\,\frac{dt}{\abs{\det t}}
\]
for special sections $W$, $\d{W}$ and for $-\Re s\gg1$ (see Proposition \ref{prop: afterfe}).
To finish the argument, we need to show that the right-hand side has analytic continuation at $s$ and compute its value at $s=\frac12$.
These steps are carried out in \cite{MR3431601}.
See \S\ref{sec: define E} where it is also explained why the restriction on $\alt{W}$ is harmless from the point of view of the relation \eqref{eq: MI'}.
\end{enumerate}

\section{A bilinear form} \label{sec: preliminary}
We begin the analysis of the main identity.
A key ingredient is the stability of the integral defining a Bessel function,
which was proved in \cite{MR3021791} following ideas of Baruch \cite{MR2192818}.

\subsection{}

We recall the main result of \cite{MR3021791} for the group $\tilde G$.
\begin{theorem}\label{thm: stable} \cite{MR3021791}
Let $K_0\in\csgr^s(G')$.
Then the integral
\[
\stint_{N'} \tilde W(\twU\animg{\wnM tn})\psi_{\tilde N}(n)^{-1}\,dn
\]
stabilizes uniformly for $\tilde W\in C(\tilde N\bs \tilde G,\psi_{\tilde N})^{K_0}$ and locally uniformly in $t\in T'$.
In particular, if $\tilde W\in\WhitG(\tilde\pi)$ with $\tilde\pi\in\Irr_{\gen,\psi_{\tilde N}}\tilde G$ then
\begin{equation} \label{eq: Bess rltn}
\stint_{N'} \tilde W(\twU\animg{\wnM tn})\psi_{\tilde N}(n)^{-1}\,dn=\Bes_{\tilde\pi}^{\psi_{\tilde N}}(\twU\animg{\wnM t})\tilde W(e)
\end{equation}
where $\Bes_{\tilde\pi}^{\psi_{\tilde N}}$ is the Bessel function of $\tilde\pi$ (see \eqref{eq: defbes}).
\end{theorem}

\begin{remark}
Note that $w_0'=w'_{U'}\wnM$. Of course, one can state the theorem above (equivalently) for
\[
\stint_{N'} \tilde W(\animg{w_0'tn})\psi_{\tilde N}(n)^{-1}\,dn.
\]
The original formulation will be slightly more convenient for computation.
\end{remark}

We will apply this to the function $\tilde W(g)=\nwhitform(W,\tilde g)$.

\begin{corollary} \label{cor: stbessel}
Let $K_0\in\csgr(G)$. Then the integral
\[
\oldhalf(W,t):=\stint_{N'}\nwhitform(W,\twU\animg{\wnM tn})\psi_{\tilde N}( n)^{-1}\,dn, \ \ t\in T'
\]
stabilizes uniformly for $W\in C(N\bs G,\psi_N)^{K_0}$ and locally uniformly in $t\in T'$.
In particular, $\oldhalf(W_s,t)$ is entire in $s\in\C$ and
if $\pi\in\Irr_{\gen}\GLnn$ and $W\in\Ind(\WhitML(\pi))$ then $\oldhalf(M^*_sW,t)$ is meromorphic in $s$.
Both $\oldhalf(W_s,t)$ and $\oldhalf(M^*_sW,t)$ are locally constant in $t$, uniformly in $s\in\C$.

Finally, if we assume that $\pi\in\Irr_{\meta,\gen}\GLnn$ and that $\tilde\pi=\desinv(\pi)$ is irreducible
then for any $\alt{W}\in\Ind(\WhitMLd(\pi))$ we have
\[
\oldhalfd(M^*\alt{W},t)=\Bes_{\tilde\pi}^{\psi_{\tilde N}^{-1}}(\twU\animg{\wnM t})\nwhitformd(M^*\alt{W},e).
\]
\end{corollary}

Another useful consequence of Theorem \ref{thm: stable} is the following.

\begin{lemma} \label{lem: bruhdec}
Let $\tilde W\in C^{\smth}(\tilde N\bs \tilde G,\psi_{\tilde N})$ and $\d{\tilde W}\in C^{\smth}(\tilde N\bs \tilde G,\psi_{\tilde N}^{-1})$.
Assume that the function $(t,n)\in T'\times N'\mapsto\tilde W(\animg{w_0'tn})$ is compactly supported.
(In particular, $\tilde W\in C_c^{\smth}(\tilde N\bs \tilde G,\psi_{\tilde N})$.)
Then the iterated integral
\[
\stint_{N'}\big(\int_{N'\bs G'}\tilde W(\tilde g)\d{\tilde W}(\tilde g\tilde u)\ dg\big)
\psi_{\tilde N}(\tilde u)\ du
\]
is well defined and is equal to
\[
\int_{T'}\int_{N'} \modulus_{B'}(t)\tilde W(\twU\animg{\wnM tn})\psi_{\tilde N}(\animg{n}^{-1})
\big(\stint_{N'}\d{\tilde W}(\twU\animg{\wnM tu})\psi_{\tilde N}(\animg{u})\ du\big)\ \ dn\ dt.
\]
\end{lemma}

\begin{proof}
Indeed, using the Bruhat decomposition we can write the left-hand side as
\[
\stint_{N'}\big(\int_{T'}\int_{N'}\tilde W(\twU\animg{\wnM tn})
\d{\tilde W}(\twU\animg{\wnM tn}\animg{u})\modulus_{B'}(t) \psi_{\tilde N}(\animg{u})\ dn\ dt\big)\ du.
\]
Note that by the assumption on $\tilde W$, the integrand is compactly supported in $t$, $n$.
Consider
\begin{multline*}
\int_{\Omega}\big(\int_{T'}\int_{N'}\tilde W(\twU\animg{\wnM tn})
\d{\tilde W}(\twU\animg{\wnM tn}\animg{u})\modulus_{B'}(t) \psi_{\tilde N}(\animg{u})\ dn\ dt\big)\ du\\=
\int_{T'}\int_{N'} \int_{\Omega}\modulus_{B'}(t)\tilde W(\twU\animg{\wnM tn})
\d{\tilde W}(\twU\animg{\wnM tn}\animg{u})\psi_{\tilde N}(\animg{u})\ du\ dn\ dt
\end{multline*}
for $\Omega\in\csgr(N')$.
By Theorem \ref{thm: stable} this is equal to
\[
\int_{T'}\int_{N'}\modulus_{B'}(t)\tilde W(\twU\animg{\wnM tn}) \big(\stint_{N'}
\d{\tilde W}(\twU\animg{\wnM tn}\animg{u})\psi_{\tilde N}(\animg{u})\ du\big)\ dn\ dt
\]
provided that $\Omega$ is sufficiently large (independently of $t$ and $n$, since they can be confined to compact sets by the assumption on $\tilde W$).
Making a change of variable $u\mapsto n^{-1}u$ we get
\[
\int_{T'}\int_{N'}\modulus_{B'}(t) \tilde W(\twU\animg{\wnM tn})\psi_{\tilde N}(\animg{n}^{-1})
\big(\stint_{N'}\d{\tilde W}(\twU\animg{\wnM tu})\psi_{\tilde N}(\animg{u})\ du\big)\ dn\ dt
\]
as required.
\end{proof}

\subsection{}\label{sec: defspecial}
In order to apply Lemma \ref{lem: bruhdec} for $\nwhitform(W_s,\cdot)$ we make a special choice of $W$.
Consider the $P$-invariant subspace $\Ind(\WhitML(\pi)){\spcl}$ of $\Ind(\WhitML(\pi))$ consisting
of functions supported in the big cell $Pw_UP=Pw_U U$.
Any element of $\Ind(\WhitML(\pi)){\spcl}$ is a linear combination of functions of the form
\begin{equation}\label{eq: spl}
W(u'mw_Uu)=\modulus_P(m)^{\frac12}W^M(m)\phi(u), \ m\in \Levi, u,u'\in U
\end{equation}
with $W^M\in\WhitML(\pi)$ and $\phi\in\swrz(U)$.
Let $\toM$ be the embedding $\toM(g)=\sm{g}{}{}{I_{\rkn}}$ of $\GLn$ into $\GLnn$.
Also let $\toLevi=\levi\circ\toM$ so that $\toLevi(g)=\left(\begin{smallmatrix}g&&\\&I_{2{\rkn}}&\\&&g^*\end{smallmatrix}\right)$.

\begin{definition} \label{def: spclW}
Let $\spclW$ be the linear subspace of $\Ind(\WhitML(\pi)){\spcl}$ generated by
$W$'s as in \eqref{eq: spl} that satisfy the additional property that the function
$(t,n)\mapsto W^\Levi(\toLevi(t\wn n))$ is compactly supported on $T'_{\GLn}\times N'_{\GLn}$,
or equivalently, that the function $W^\Levi\circ\toLevi$ on $\GLn$
is supported in the big cell $B'_{\GLn}\wn N'_{\GLn}$ and its support is compact modulo $N'_{\GLn}$.
\end{definition}

This space is nonzero, see the proof of Lemma~\ref{lem: Bnonvanish} below.
It is also invariant under $\toG(T')\ltimes N$.

\begin{lemma} \label{L: Iexp2}
For any $W\in\spclW$, the function
$\nwhitform(W_s, \twU\animg{\wnM tn})$ is compactly supported in $t\in T'$ and $n\in N'$ uniformly in $s\in\C$.
\end{lemma}

\begin{proof}
Since $\Phi*W\in\spclW$ whenever $W\in\spclW$ (for any $\Phi\in\swrz(F^\rkn)$), it is enough to show (in view of the definition of $\nwhitform$,
upon replacing the domain of integration in \eqref{eq: defwhitform} by $\toG(w'_{U'})V_U\toG(w'_{U'})^{-1}$)
that for any $W\in\spclW$, the function $(v,t,n)\mapsto W(\gamma \toG(w'_{U'})v\toG(\wnM tn))$ is compactly supported in $V_U\times T'\times N'$.
Write $t=\levi'(t')$ and $n\in N'$ as $n=\levi'(n_1')n_2$ with $n_1\in N'_{\GLn}$ and $n_2\in U'$.
Also let $m'=\wn t'n_1'\in\GLn$ and $m=\toG(\levi'(m'))=\toG(\wnM t\levi'(n_1'))$. Then
\[
W(\gamma \toG(w'_{U'})v\toG(\wnM tn))=W(w_Uvm\toG(n_2))=W(\toLevi((m')^*)w_U m^{-1}vm\eta(n_2)).
\]
It follows immediately from the definition of the space $\spclW$ that
 this function is compactly supported in $(v,n_1, n_2, t')\in V_U\times N'_{\GLn}\times U'\times T'_{\GLn}$.
 \end{proof}

\subsection{} \label{sec: defBil}

For $W \in \spclW$ and $\d{W}\in \Ind(\WhitMLd(\d\pi))$ we define
\[
\Bil(W,\d{W},s):=\stint_{N'}\big(\int_{N'\bs G'}\nwhitform(W_s,\tilde g)\nwhitformd(\d{W}_{-s} ,\tilde g\tilde u)\ dg\big)
\psi_{\tilde N}(\tilde u)\ du.
\]
By Lemma \ref{L: Iexp2} and the proof of Lemma \ref{lem: bruhdec}, $\Bil(W,\d{W},s)$ is an entire function of $s$ and we have
\begin{equation} \label{eq: afterbruhat}
\Bil(W ,\d{W},s)=\int_{T'}\oldhalf(W_s,t)\oldhalfd(\d{W}_{-s},t)\modulus_{B'}({t})\,dt
\end{equation}
for any $W \in \spclW$ and $\d{W}\in \Ind(\WhitMLd(\pi))$.

Assume that $\pi\in\Irr_{\gen,\meta}\GLnn$ and $\tilde\pi=\desinv(\pi)$ is irreducible.
Then for any $W \in \spclW$ and $\alt{W}\in \Ind(\WhitMLd(\pi))$,
\begin{equation} \label{eq: LHSofMI}
\text{the left-hand side of the Main Identity \eqref{eq: MI'} is }\Bil(W,M(\frac12)\alt{W},\frac12).
\end{equation}

In the next section we will obtain a more explicit form of $\Bil(W,\d{W},s)$ for $W\in\spclW$.

\section{Further analysis}

\subsection{The function $\nwhitform(W_s,\animg{g})$} \label{sec: expwhitform}

At this point it is necessary to explicate $\nwhitform(W_s,\cdot)$ on the big cell.

We start with explicating $\nwhitform(W,e)$, which is easier and useful for the right-hand side of \eqref{eq: main2}.

Fix an element $\auxthree\in V$ of the form $\toU_\Levi(X)$ where $X\in\Mat_{\rkn}$ and the last row of $X$ is $-\xi_\rkn$.
For any $W\in C^{\smth}(N\bs G,\psi_N)$ define
\begin{equation}\label{eq: defMint}
\Mint(W):=\int_{V_\gamma\bs \Vm} W(\gamma v\auxthree )\psi_{\Vm}(\auxthree^{-1} v \auxthree )^{-1}\, dv.
\end{equation}
Since $\psi_{\Vm}(\auxthree^{-1} v \auxthree )=\psi_{\Vm}(v)\psi(v_{\rkn,2\rkn+1})$, this expression is clearly independent of the choice of $\auxthree$ as above.
Moreover, we have
\begin{multline}\label{eq: defMintvar}
\Mint(W)=\int_{(V_\gamma\rtimes\toG(N'))\bs \kgrp} W(\gamma n\auxthree )\psi_{\kgrp}(\auxthree^{-1} n \auxthree )^{-1}\, dn
\\=\int_{\auxthree^{-1}(V_\gamma\rtimes\toG(N'))\auxthree\bs \kgrp} W(\gamma \auxthree n )\psi_{\kgrp}( n )^{-1}\, dn.
\end{multline}
The integrand in the first equation is invariant under $\toG(N')$ since the character
$\psi_{\kgrp}(\auxthree^{-1} \toG(u) \auxthree)$ is trivial on $U'$ and agrees with $\psi_{N'_{\Levi'}}$ on $N'_{\Levi'}$.

\begin{lemma}\label{L: Mint}
For any $W\in C^{\smth}(N\bs G,\psi_N)$,
the integrand in \eqref{eq: defMint} is compactly supported on $V_\gamma\bs \Vm$ and we have $\nwhitform(W,e)=\Mint(W)$.
\end{lemma}

\begin{remark}
Note that by Lemma \ref{lem: equnwhit}, the linear form $W\mapsto\nwhitform(W,e)$ is $(\kgrp,\psi_{\kgrp})$-equivariant.
By \eqref{eq: defMintvar} the same is true for $\Mint(W)$.
\end{remark}

\begin{proof}
The support condition follows from Lemma \ref{lem: support subsym}. 
By \eqref{eq: defwhitform} we have
\[
\whitform(W,\Phi,e)=\int_{V_\gamma\bs V} W(\gamma v)\wevinv(v)\Phi(\xi_{\rkn})\,dv.
\]
We can write this as
\[
\int_{\Vp}\int_{V_\gamma\bs \Vm} W(\gamma vc)\wevinv(vc)\Phi(\xi_{\rkn})\,dv\ dc
\]
which by \eqref{eq: weilH1}--\eqref{eq: weilH3} is equal to
\[
\int_{\Vp}\int_{V_\gamma\bs \Vm} W(\gamma vc)\psi_{\Vm}(v)^{-1}\psi(v_{\rkn,2\rkn+1})^{-1}\Phi(\underline{c}+\xi_{\rkn})\,dv\ dc.
\]
Changing variables in $c$ and $v$ we can rewrite this as
\begin{multline*}
\int_{\Vp}\int_{V_\gamma\bs \Vm} W(\gamma v\auxthree c)\Phi(\underline{c})\psi_{\Vm}(v)^{-1}\psi(v_{\rkn,2\rkn+1})^{-1}\,dv\ dc=\\
\int_{V_\gamma\bs \Vm} (\Phi*W)(\gamma v\auxthree)\psi_{\Vm}(v)^{-1}\psi(v_{\rkn,2\rkn+1})^{-1}\,dv\ dc.
\end{multline*}
This is $\Mint(\Phi*W)$ as $\psi_{\Vm}(\auxthree^{-1} v \auxthree )=\psi_{\Vm}(v)\psi(v_{\rkn,2\rkn+1})$.
\end{proof}

We now explicate $\nwhitform(W_s,\cdot)$ on the big cell. By Lemma \ref{lem: equnwhit} it is enough to consider the element $w'_{U'}$.

\begin{lemma} \label{lem: onw'U}
Let $\pi\in\Irr_{\gen}\Levi$.
Then for $\Re s\gg_{\pi}1$ and any $W\in\Ind(\WhitML(\pi))$ we have
\begin{multline} \label{eq: nwhitbcg}
\nwhitform(W_s,\twU)=
\beta_{\psi^{-1}}(w'_{U'})\int_{V_U}W_s(w_Uv)\psi_U(v)^{-1}\,dv\\=
\beta_{\psi^{-1}}(w'_{U'})\int_{\altV_{\Levi}\bs \Vm}W_s(w_Uv)\psi_{\Vm}(v)^{-1}\ dv.
\end{multline}
\end{lemma}

\begin{remark}
It is easy to see that both sides of \eqref{eq: nwhitbcg} are $(\Vm\rtimes\toG(N'_{\Levi'}),\psi_{\Vm}\psi_{N'_{\Levi'}}^{-1})$-equivariant in $W$.
\end{remark}

\begin{proof}
From \eqref{eq: defwhitform},
\[
\whitform(W_s,\Phi,\twU)=
\int_{V_\gamma\bs V}W_s(\gamma v \toG(w'_{U'}))\wevinv(v\twU)\Phi(\xi_{\rkn})\ dv.
\]
Make a change of variable $v\mapsto \toG(w'_{U'})v\toG(w'_{U'})^{-1}$.
Note that $V_\gamma=\toG(w'_{U'})V_{\Levi}\toG(w'_{U'})^{-1}$, $\toG(w'_{U'})$ normalizes $V$, and $V=V_{\Levi}\ltimes V_U$.
By \eqref{eq: weilext2} we infer that
\[
\whitform(W_s,\Phi,\widetilde{w'_{U'}})=\int_{V_U}W_s(w_Uv)\wevinv(\widetilde{w'_{U'}}v)\Phi(\xi_{\rkn})\ dv
\]
and using \eqref{eq: weil3} we get
\[
\beta_{\psi^{-1}}({w'_{U'}})\int_{V_U}\big(\int_{F^{\rkn}} W_s(w_Uv)\wevinv(v)\Phi(Y)\psi(Y_1)^{-1}\,dY \big)\ dv.
\]
We claim that the double integral converges absolutely when $\Re s$ is large enough.
Note that $V_U=\{\toU(\sm{x}{y}{}{\startran{x}}),\ x\in \Mat_\rkn, y\in \symspace_{\rkn}\}$
and hence $\abs{\wevinv(v)\Phi(Y)}=\abs{\Phi(Y)}$ for $v\in V_U$. Thus,
\[
\int_{V_U}\int_{F^{\rkn}} \abs{W_s(w_Uv)\wevinv(v)\Phi(Y)\psi(Y_1)^{-1}}\,dY \ dv=
\widehat{\abs{\Phi}}(0)\int_{V_U} \abs{W_s(w_Uv)} \ dv.
\]
The integration on the right-hand side is absolutely convergent when $\Re s\gg_{\pi} 1$.
This follows from the convergence of $\int_U \abs{W_s(w_Uu)} \ du$ and Remark \ref{rem: partint}.

For $c\in \Vp$ such that $\underline{c}=Y\in F^{\rkn}$, we have $W_s(w_U cg)=\psi(Y_1)^{-1}W_s(w_Ug)$
by the equivariance of $W$, and $\wevinv(c)\Phi(0)=\Phi(Y)$ by \eqref{eq: weilH1}. Thus,
\[
\beta_{\psi^{-1}}({w'_{U'}})^{-1}\whitform(W_s,\Phi,\widetilde{w'_{U'}})=\int_{V_U}\int_{\Vp} W_s(w_Ucv)\wevinv(cv)\Phi(0) \,dc\ dv.
\]
Since $\Vp$ normalizes $V_U$ and the double integral converges we can write the last integral as
\begin{multline*}
\int_{V_U}\int_{\Vp}W_s(w_Uvc)\wevinv(vc)\Phi(0)\ dc\ dv\\=
\int_{V_U}\int_{\Vp}W_s(w_Uvc)\psi_{\Vm}(v)^{-1}\Phi(\underline{c})\ dc\ dv=
\int_{V_U}(\Phi*W_s)(w_Uv)\psi_{\Vm}(v)^{-1}\ dv.
\end{multline*}
This give the first equality in \eqref{eq: nwhitbcg}, and the second one follows from it since the integrand is left $\altV_{\Levi}$-invariant.
\end{proof}

\begin{remark} \label{rem: genbc}
One can show that for any $W\in C^{\smth}(N\bs G,\psi_N)$ we have
\begin{equation} \label{eq: nwhitbc}
\nwhitform(W,\widetilde{w'_{U'}})=\beta_{\psi^{-1}}(w'_{U'})
\int_{V''\bs \Vm}\big(\int_{\altV_{\Levi}\bs V''}W(w_Uvx)\psi_{\Vm}(vx)^{-1}\,dv\big)\,dx
\end{equation}
where on the right-hand side $V''$ is the preimage of the center of the Heisenberg group under $v\mapsto \Hei{v}$
and in both the inner and outer integral the integrand is compactly supported, thus convergent.
We will not give the details since we are not going to use this result.
\end{remark}

\subsection{}
We can now explicate $\oldhalf(W_s,t)$ for $\Re s\gg1$.

\begin{lemma}\label{L: Iexp}
Let $\pi\in \Irr_{\gen}\Levi$. For $\Re s\gg_{\pi}1$ and any $W\in \Ind(\WhitML(\pi))$, $t\in T'$ we have the identity
\begin{equation}\label{eq: Iexp}
\oldhalf(W_s,t)=
\wgt{t}^{{\rkn}-\frac12}\beta_{\psi^{-1}}(w'_{U'})\beta_{\psi^{-1}}(\wnM t)
\int_{\kgrpq}W_s(w_U\toG( \wnM t)v)\psi_{\kgrp}(v)^{-1}\,dv
\end{equation}
where the right-hand side is absolutely convergent.
\end{lemma}

\begin{remark}
Clearly, both sides of \eqref{eq: Iexp} are $(\kgrp,\psi_{\kgrp})$-equivariant in $W$.
\end{remark}

We first need the following convergence result.

\begin{lemma}\label{L:conv}
Let $\pi\in\Irr_{\gen}\Levi$. Then for $\Re s\gg_{\pi}1$ we have
\begin{subequations}
\begin{gather}
\label{conv: row}\int_{N_\GLnn^t\cap\mira}\int_U\abs{W_s \left(\levi(n)w_Uug \right)}\,du\,d n<\infty,\\
\label{conv: column} \int_{N_\GLnn^t\cap\mira^*}\int_U\abs{W_s \left(\levi(n)w_Uug \right)}\,du\,d n<\infty,
\end{gather}
\end{subequations}
for any $W\in\Ind(\WhitML(\pi))$, $g\in G$.
\end{lemma}

\begin{proof}
We may assume without loss of generality that $g=e$.
Denote by $\levi(a(x))$ the torus part in the Iwasawa decomposition of $x\in G$.
Let $\mnrs_i(x)$ be the set of non-zero $i\times i$ minors in the last $i$ rows of $x$ (with the convention that $\mnrs_0(x)=\{1\}$).
Then $\mnrs_i(\levi(a(x)))$ is (obviously) a singleton and its absolute value is $\max_{\delta\in\Delta_i(x)}\abs{\delta}$.

For any $n\in N_\GLnn^t$ and $u\in U$ consider
\[
a(\levi(n)w_Uu)=\diag(a_1, \dots,a_{2{\rkn}}).
\]
Notice that
\begin{equation}\label{eq: sameminor}
\prod_{i=1}^{2{\rkn}}\abs{a_i}=\abs{\det(a(\levi(n)w_Uu))}=\abs{\det(a(w_Uu))}\le1
\end{equation}
and for all $1\le j\le i\le 2\rkn$ we have $n_{i,j}\in \{0\}\cup\Delta_{i-1}(\levi(n)w_Uu)$
(since $n_{i,j}$ is an entry in the adjugate of $n^*$).
Hence
\begin{equation} \label{eq: uprbndnij}
\abs{n_{i,j}}\le\prod_{k=1}^{i-1}\abs{a_k}^{-1}.
\end{equation}
We will show below that for all $n\in N_\GLnn^t\cap (\mira\cup\mira^*)$ and $u\in U$ such that $W(\levi(n)w_Uu)\ne0$ we have
\begin{equation}\label{eq: boundnentry}
 \abs{\det(a(w_Uu))} \ll_{W} \abs{a_i}\ll_{W} \abs{\det(a(w_Uu))}^{2-2\rkn},\ \ \ i=1,\dots 2\rkn.
\end{equation}

Assuming this, we first show the lemma.
It follows from \eqref{eq: sameminor} and \eqref{eq: boundnentry} that there exists $\lambda\in\R$, depending only on $\pi$ such that
\begin{equation}\label{conv: leviind}
\abs{W_s(\levi(n)w_Uu)}\ll_W\abs{\det(a(w_Uu))}^{\Re s-\lambda}
\end{equation}
for any $u\in U$ and $n\in N_\GLnn^t\cap(\mira\cup\mira^*)$. Together with \eqref{eq: uprbndnij} and \eqref{eq: boundnentry} we obtain
$\abs{n_{i,j}}\ll_W  \abs{\det(a(w_Uu))}^{1-i}$ whenever $W(\levi(n)w_Uu)\ne0$. Thus,
\[
\int_{N_\GLnn^t\cap\mira}\int_U\abs{W_s(\levi(n)w_Uu)}\ du\ dn
\ll_W \int_U\abs{\det(a(w_Uu))}^{\Re s-\lambda-\sum_{i=1}^{2\rkn-1}i^2}\ du
\]
(and similarly for $\int_{N_\GLnn^t\cap\mira^*}$)
where the last integral is finite when $\Re s\gg_\pi 1$ by a standard result on intertwining operators.

We are left to prove \eqref{eq: boundnentry}. First, consider the case when $n\in N_\GLnn^t\cap\mira$.
Then the $2{\rkn}$-th row of $\levi(n)w_Uu$ is $\xi_{4\rkn}$.
Thus, $\mnrs_{2\rkn+1}(\levi(n)w_Uu)\subset\mnrs_{2\rkn}(\levi(n)w_Uu)$ and hence
\[
\abs{a_{2{\rkn}}}\prod_{i=1}^{2{\rkn}}\abs{a_i}^{-1}\le\prod_{i=1}^{2{\rkn}}\abs{a_i}^{-1},
\]
or equivalently $\abs{a_{2{\rkn}}}\le1$.
Thus, if $W(\levi(n)w_Uu)\ne0$ then $\abs{a_1}\ll_W \cdots\ll_W\abs{a_{2{\rkn}}}\le 1$.
The relation \eqref{eq: boundnentry} follows (since $1\leq \abs{\det(a(w_Uu))}^{2-2\rkn}$).

Next, consider the case $n\in N_\GLnn^t\cap\mira^*$. The last row of
$\levi(n)w_Uu$ is the same as the last row of $w_Uu$. Thus, $\abs{a_1}^{-1}$ is bounded above by the norm of the
entries in $u$, which in turn is bounded above by $\abs{\det(a(w_Uu))}^{-1}$ (since each non-zero entry of $u$ belongs to $\mnrs_{2\rkn}(w_Uu)$).
Again, if $W(\levi(n)w_Uu)\ne0$ then $\abs{a_1}_W\ll \cdots\ll_W\abs{a_{2{\rkn}}}$. From \eqref{eq: sameminor}
we conclude that $\abs{\det(a(\levi(n)w_Uu))}\ll\abs{a_1}\ll_W\abs{a_i}$ for all $i$.
On the other hand,
\begin{multline*}
\abs{\det(a(w_Uu))}^{2\rkn-2}=\abs{a_1}^{2\rkn-2}\abs{a_2\dots a_{2\rkn-1}}^{2\rkn-2}\abs{a_{2\rkn}}^{2\rkn-2}\ll_W
\abs{a_2\dots a_{2\rkn-1}}^{2\rkn-1}\abs{a_{2\rkn}}^{2\rkn-2}\\=
\abs{a_2\dots a_{2\rkn}}^{2\rkn-1}\abs{a_{2\rkn}}^{-1}\le\abs{a_{2\rkn}}^{-1}
\end{multline*}
and the relation \eqref{eq: boundnentry} follows.
\end{proof}

\begin{remark}
From the proof it is clear that we can take a uniform region of convergence $\Re s\gg1$ for all unitarizable $\pi$.
\end{remark}

Lemma \ref{L: Iexp} is an immediate consequence of Lemmas \ref{lem: onw'U}, \ref{lem: equnwhit} and \ref{L:conv}.
Indeed, the right-hand side of \eqref{eq: Iexp} can be viewed as a `partial integration'
of the double integral \eqref{conv: row} for $g=\toG(\wnM t)$.
Thus, by Remark \ref{rem: partint}, it is absolutely convergent for $\Re s\gg_\pi1$.
Meanwhile by Lemma  \ref{lem: equnwhit},
\begin{multline*}
\oldhalf(W_s,t)=\int_{N'}\nwhitform(W_s,\twU\animg{\wnM tn})\psi_{\tilde N}( n)^{-1}\,dn
\\=\wgt{t}^{-\frac12}\beta_{\psi^{-1}}(\wnM t)\int_{N'}\nwhitform(W_s(\cdot \toG(\wnM tn)),\twU)\psi_{\tilde N}( n)^{-1}\,dn
\end{multline*}
assuming the integral converges. From \eqref{eq: nwhitbcg} this is
\[
\wgt{t}^{-\frac12}\beta_{\psi^{-1}}(\wnM t)\beta_{\psi^{-1}}(w'_{U'})\int_{N'}\int_{V_U}W_s(w_Uv\toG(\wnM tn))\psi_U(v)^{-1}\psi_{\tilde N}( n)^{-1}\,dv\,dn.
\]
Conjugating $v$ over $\toG(\wnM t)$ and combining the two integrals (which by the above, converge) we get Lemma \ref{L: Iexp}.
(We note that the integrand on the right-hand side of \eqref{eq: Iexp} is indeed left $\altV_{\Levi}$-invariant.)

\subsection{}
We now go back to the bilinear form $\Bil$ defined in \S\ref{sec: defBil}.
Recall the definition of the space $\spclW$ in \S\ref{sec: defspecial}.
\begin{lemma}\label{lem: Iexp2}
For $W\in\spclW$ the integrand on the right-hand side of \eqref{eq: Iexp} is compactly supported in $t, v$ uniformly in $s$
(i.e., the support in $(t,v)$ is contained in a compact set which is independent of $s$).
In particular, the identity \eqref{eq: Iexp} holds for all $s\in\C$.
\end{lemma}

\begin{proof}
Suppose that $W\in\spclW$ is of the form \eqref{eq: spl}.
Then for $n\in N'_{\GLn}$, $u\in U$ and $t\in T'_{\GLn}$, we have $\toG(\levi'( \wn t n))\in \Levi$ and
\[
W\left(w_U\toG(\levi'( \wn t n))u\right)=\modulus_P(\toG(t))^{-\frac12}W^\Levi(\toLevi((\wn t n)^*))\phi(u).
\]
The lemma follows from the definition of $\spclW$.
\end{proof}

\begin{lemma}\label{L: defnewBil}
Let $\pi\in\Irr_{\gen}\Levi$. Then for $-\Re s\gg1$ we have
\begin{multline*}
\Bil(W ,\d{W},s)=\int_U\int_{\kgrpq}\int_{N'_{\GLn}\bs\GLn}
W _s\left(\toLevi(g)w_Uv\right)\d{W}_{-s}\left(\toLevi(g)w_Uu\right)\\\modulus_P(\toLevi(g))^{-1}\abs{\det g}^{1-{\rkn}}
\psi_{\kgrp}(v)^{-1}\psi_U(u) \,dg\,dv\,du
\end{multline*}
for any $W\in\spclW$, $\d{W}\in\Ind(\WhitMLd(\d\pi))$, with the integral being absolutely convergent.
\end{lemma}

\begin{proof}
Suppose that $-\Re s\gg1$ and $W\in\spclW$. Then by \eqref{eq: afterbruhat}, Lemmas \ref{L: Iexp} and \ref{lem: Iexp2}, and the fact that
$\modulus_{B'}(t)=\modulus_{B'_{\Levi'}}(t)\wgt{t}^{{\rkn}+1}$, $\Bil(W ,\d{W},s)$ equals
\begin{multline*}
\int_{T'_{\GLn}}\int_{\toG(N'_{\Levi'})\ltimes U}\int_{\toG(N'_{\Levi'})\ltimes U}
W _s\left(\toLevi(( \wn t)^*) w_Uv_1\right)
\d{W}_{-s}\left(\toLevi(( \wn t)^*) w_Uv_2 \right)
\\ \abs{\det t}^{3{\rkn}}\modulus_{B'_{\GLn}}(t)\psi_{\kgrp}(v_1)^{-1}\psi_{\kgrp}(v_2) \,dv_1\,dv_2\,dt
\end{multline*}
where the integral is absolutely convergent. Here we used $\toG(N'_{\Levi'})\ltimes U$ as a section of $\kgrpq$.

Making a change of variable $v_1\mapsto v_2v_1$
and writing $v_2=\toG(\levi'(n))u$ with $n\in N'_{\GLn}$ and $u\in U$,
we get that $\Bil(W,\d{W},s)$ is equal to
\begin{multline*}
\int_{T'_{\GLn}}\int_{N'_{\GLn}}\int_U\int_{\toG(N'_{\Levi'})\ltimes U}
W _s\left(\toLevi( (\wn tn)^*)w_U uv_1\right)
\d{W}_{-s}\left(\toLevi((\wn tn)^*) w_U u \right)
\\ \abs{\det t}^{3{\rkn}}\modulus_{B'_{\GLn}}(t)\psi_{\kgrp}(v_1)^{-1} \,dv_1\,du\,dn\,dt.
\end{multline*}
Make a further change of variable $v_1\mapsto u^{-1}v_1$. It remains to use the Bruhat decomposition for $\GLn$ and to note that
$\modulus_P(\toLevi(g))=\abs{\det g}^{2{\rkn}+1}$ for any $g\in\GLn$ and that as a function of $v$, the integrand in the statement of the lemma
is left $\altV_{\Levi}$-invariant.
\end{proof}

Define when convergent
\[
\semipair{W}{\d{W}}:=\int_{N'_{\GLn}\bs\GLn}W(\toLevi(g))\d{W}(\toLevi(g))\modulus_P(\toLevi(g))^{-1}\abs{\det g}^{1-{\rkn}}\ dg.
\]
Then we get for any $W\in\spclW$, $\d{W}\in\Ind(\WhitMLd(\d\pi))$ and when $-\Re s\gg1$
\begin{equation}\label{eq: explicitBil}
\Bil(W ,\d{W},s)=\int_U\int_{\kgrpq}\semipair{W_s(\cdot w_Uv)}{\d{W}_{-s}(\cdot w_Uu)}\psi_{\kgrp}(v)^{-1}\psi_U(u) \,dv\,du.
\end{equation}

\begin{remark}
If $\pi$ is unitarizable then the integral defining $\semipair{W}{\d{W}}$ is absolutely convergent
for any $(W,\d{W})\in \Ind(\WhitML(\pi))\times\Ind(\WhitMLd(\d\pi))$ by \cite[Lemma 1.2]{MR3220931}.
\end{remark}

\subsection{}
We will need a non-vanishing result which follows from Theorem \ref{thm: nonvanishingBesselfunction} of Appendix \ref{sec: nonvanishing}.

\begin{lemma} \label{lem: Bnonvanish}
Assume that $\pi\in\Irr_{\gen,\meta}\Levi$ and $\tilde\pi=\desinv(\pi)$ is irreducible and tempered.
Then the bilinear form $\Bil(W ,M(\frac12)\alt{W},\frac12)$ does not vanish identically on $\spclW\times\Ind(\WhitMLd(\pi))$.
\end{lemma}

\begin{proof}
Since the image of the restriction map
$\WhitML(\pi)\rightarrow C(N_\Levi\bs\levi(\mira),\psi_{N_\Levi})$ contains
$\swrz(N_\Levi\bs\levi(\mira),\psi_{N_\Levi})$,
it follows that for any $\varphi\in \swrz(T'_{\GLn}\times N'_{\GLn})$ and $\phi\in\swrz(U)$ there exists
$W\in\Ind(\WhitML(\pi)){\spcl}$ (necessarily in $\spclW$) of the form \eqref{eq: spl} such that $\varphi(t,n)=W^\Levi(\toLevi((t\wn n)^*))$.
It follows from Lemma \ref{lem: Iexp2} that the linear map $\spclW\rightarrow \swrz(T')$ given by
$W\mapsto\oldhalf(W_{\frac12},\cdot)$ is onto.
Therefore, by \eqref{eq: afterbruhat}, the lemma amounts to the non-vanishing of
the linear form $\alt{W}\mapsto\oldhalfd(M^*\alt{W},t)$ on $\Ind(\WhitMLd(\pi))$ for some $t\in T'$.
This follows from Theorem \ref{thm: nonvanishingBesselfunction} and the last part of Corollary \ref{cor: stbessel}.
\end{proof}

\subsection{}
By \eqref{eq: LHSofMI}, Lemma \ref{L: Mint} and Lemma \ref{lem: Bnonvanish}, we conclude
\begin{corollary} \label{cor: Bnonvanish}
Suppose that $\pi\in\Irr_{\meta,\temp}\Levi$ is \good\
and $\tilde\pi=\desinv(\pi)$ is tempered.
Then
\[
\Bil(W ,M(\frac12)\alt{W},\frac12)=c_\pi\Mint(M^*W)\Mintd(M^*\alt{W})
\]
for all $W \in\spclW$ and $\alt{W}\in\Ind(\WhitMLd(\pi))$. Moreover, the linear form $\Mint(M^*W)$ does not vanish identically on $\spclW$.
\end{corollary}

In other words, taking into account the reduction step of \S\ref{sec: firstred} we have reduced Theorem \ref{thm: main} to the following statement.

\begin{proposition}\label{prop: main2}
Assume $\pi\in\Irr_{\meta,\temp}\Levi$ is \good\ and $\tilde\pi=\desinv(\pi)$ is tempered.
Then for any $W \in\spclW$ and $\alt{W}\in\Ind(\WhitMLd(\pi))$ we have
\begin{equation}\label{eq: main2}
\Bil(W ,M(\frac12)\alt{W},\frac12)=\eps_\pi\Mint(M^*W)\Mintd(M^*\alt{W})
\end{equation}
\end{proposition}

The proposition will eventually be proved in \S\ref{sec: define E} after further reductions,
using the results of  \cite{MR3431601} and \cite{MR3220931}.

\section{Applications of functional equations} \label{sec: FEnewBil}

In this section we will use the functional equations established in \cite{MR3220931} to analyze $\Bil(W ,M(s)\alt{W},s)$.

\subsection{}

We first apply a functional equation proved in \cite[Appendix B]{MR3220931}.
To that end we introduce a variant of $\Bil$. We define $\newBil(W,\d{W},s)$ to be the right hand side of
\eqref{eq: explicitBil} whenever the integral defining $\semipair{\cdot}{\cdot}$ and the double integrals over $\kgrpq$ and $U$ are absolutely convergent.

Clearly for $g\in \GLn$, with $\abs{\det g}=1$
\[
\semipair{W(\cdot\toLevi(g))}{\d{W}(\cdot\toLevi(g))}=\semipair{W}{\d{W}}.
\]
Thus, $\newBil(W,\d{W},s)$ is equal to
\begin{multline}\label{eq: otherBil}
\int_{N'_{\GLn}}\int_U\int_U\semipair{W_s\left(\cdot\toLevi(n)w_U v\right)}{\d{W}_{-s}\left(\cdot w_U u\right)}\psi_U(v)^{-1}\psi_U(u)
\psi_{N'_{\GLn}}(n)  \,dv\,du\,dn
\\=\int_{N'_{\GLn}}\int_U\int_U\semipair{W_s\left(\cdot w_U v\right)}{\d{W}_{-s}\left(\cdot\toLevi(n) w_U u\right)}\psi_U(v)^{-1}\psi_U(u)
\psi_{N'_{\GLn}}(n)^{-1}  \,dv\,du\,dn
\\=\int_{\kgrpq}\int_U\semipair{W_s\left(\cdot w_U v\right)}{\d{W}_{-s}\left(\cdot w_U v_2\right)}
\psi_U(v)^{-1}\psi_{\kgrp}(v_2)  \,dv\,dv_2.
\end{multline}

By \eqref{eq: explicitBil}, for any $\pi\in\Irr_{\gen}\Levi$, $W\in\spclW$, $\d{W}\in\Ind(\WhitMLd(\d\pi))$ and $-\Re s\gg1$ we have
\[
\newBil(W,\d{W},s)=\Bil(W ,\d{W},s).
\]
On the other hand, we have the following result.
\begin{proposition}(\cite[Appendix B]{MR3220931}) \label{prop: adjrltn}
Let $\pi\in \Irr_{\temp}\Levi$. Then
\begin{enumerate}
\item For $\Re s\gg1$, $\newBil(W,\d{W},s)$ is well defined for any $W\in \Ind(\WhitML(\pi))$, $\d{W}\in\Ind(\WhitMLd(\d\pi)){\spcl}$.
\item For $-\Re s\gg1$, $\newBil(W,\d{W},s)$ is well defined for any $W\in \Ind(\WhitML(\pi)){\spcl}$, $\d{W}\in\Ind(\WhitMLd(\d\pi))$.
\item For $-\Re s\gg1$ we have
\[
\newBil(W, M(s)\alt{W},s)=\newBil(M(s)W,\alt{W},-s)
\]
for any $W\in \Ind(\WhitML(\pi)){\spcl}$, $\alt{W}\in\Ind(\WhitMLd(\pi)){\spcl}$.
\end{enumerate}
\end{proposition}

Combined with the above we get

\begin{corollary} \label{cor: FEBil}
For $-\Re s\gg1$ we have
\[
\Bil(W, M(s)\alt{W},s)=\newBil(M(s)W,\alt{W},-s)
\]
for any $W\in\spclW$, $\alt{W}\in\Ind(\WhitMLd(\pi)){\spcl}$.
\end{corollary}

\subsection{An identity of Whittaker functions on $\GL_{2{\rkn}}$} \label{sec: GL_n funcequ}

A key fact in the formal argument for the case ${\rkn}=1$ in \cite[\S7]{1401.0198}
was that for any unitarizable $\pi\in\Irr_{\gen}\GL_2$, the expression
\[
\int_{F^*} W(\sm{t}{}{}{1})\d{W}(\sm{t}{}{}{1})\ dt
\]
defines $\GL_2$-invariant bilinear form on $\WhitM(\pi)\times\WhitMd(\d\pi)$.

In the general case we encountered (in the definition of $\semipair{\cdot}{\cdot}$) a similar integral
\[
A_{\rkn}(W,\d{W})=\int_{N'_{\GLn}\bs \GLn} W\left(\toM(g)\right)\d{W}\left(\toM(g)\right) \abs{\det g}^{1-{\rkn}}\,dg
\]
where $W\in\WhitM(\pi)$ and $\d{W}\in\WhitMd(\d\pi)$.

Let $\zerocol$ be the subspace of $\Mat_\rkn$ consisting of the matrices whose first column is zero.
Note that for any $g\in\GLn,X\in\Mat_\rkn,Y\in \zerocol,n_1,n_2\in N'_{\GLn}$ we have (with $\toU_{\GLnn}$, $\toM$ and $\toMd$ defined in \S\ref{sec: elements})
\begin{multline} \label{eq: Arknequ}
A_{\rkn}(\pi(\toM(g)\toU_{\GLnn}(X+Y)\toMd(n_1))W,\d{\pi}(\toM(g)\toU_{\GLnn}(X)\toMd(n_2))\d{W})=\\
\psi_{N'_{\GLn}}(n_1n_2^{-1}) \abs{\det g}^{\rkn-1}A_{\rkn}(W,\d{W}).
\end{multline}
We also have the following relation.
Let $w_{2{\rkn},{\rkn}}:=\sm{}{I_{\rkn}}{I_{\rkn}}{}$.

\begin{theorem}(\cite[Theorem~1.3]{MR3220931})  \label{T: fe}
Let $\pi\in\Irr_{\temp}\GLnn$.
Then for any $W\in\WhitM(\pi)$, $\d{W} \in \WhitMd(\d\pi)$ we have
\begin{multline} \label{eq: FE}
A_{\rkn}(W,\d{W})
\\=\int_{\zerocol}\int_{\zerocol}\int_{N'_{\GLn}\bs \GLn}
W \left(w_{2{\rkn},{\rkn}}\toM(g)\toU_{\GLnn}(X)\right)
\d{W} \left(w_{2{\rkn},{\rkn}}\toM(g)\toU_{\GLnn}(Y)\right)
\abs{\det g}^{{\rkn}-1}\,dg\,d X\,d Y.
\end{multline}
The integrals on both sides are absolutely convergent.
\end{theorem}

Note that it is not a priori clear that the right-hand side of \eqref{eq: FE} satisfies \eqref{eq: Arknequ} for $X\in \Mat_{\rkn}-\zerocol$.

We can slightly rephrase Theorem~\ref{T: fe} as follows.
First, we write the right-hand side of \eqref{eq: FE} as
\[
\int_{\zerocol}\int_{\zerocol}\int_{N'_{\GLn}\bs \GLn}
W \left(\toMd(g)w_{2{\rkn},{\rkn}}\toU_{\GLnn}(X)\right)
\d{W} \left(\toMd(g)w_{2{\rkn},{\rkn}}\toU_{\GLnn}(Y)\right)
\abs{\det g}^{{\rkn}-1}\,dg\,d X\,d Y.
\]
We may multiply $w_{2{\rkn},{\rkn}}$ on the left by any element $x\in\toMd(\GLn)$ with $\abs{\det x}=1$.
In particular, we take $w_{2{\rkn},{\rkn}}':=\sm{}{I_{\rkn}}{\wn}{}$ instead of $w_{2{\rkn},{\rkn}}$.
Let $\auxauxone\in\toU_\GLnn(\one_{1,1}+\zerocol)$ where $\one_{1,1}$ is the matrix in $\Mat_\rkn$ with $1$ in the upper left corner and
zero elsewhere.
\label{sec: bepsilon1}
Since
\[
A_{\rkn}(\pi(\auxauxone)W,\d\pi(\auxauxone) \d{W})=A_{\rkn}( W,  \d{W}),
\]
we infer that $A_{\rkn}(W,\d{W})$ equals
\[
\int_{\zerocol}\int_{\zerocol}\int_{N'_{\GLn}\bs \GLn}
W\left(\toMd(g)w_{2{\rkn},{\rkn}}'\toU_{\GLnn}(X)\auxauxone\right)
\d{W}  \left(\toMd(g)w_{2{\rkn},{\rkn}}'\toU_{\GLnn}(Y)\auxauxone\right)
\abs{\det g}^{{\rkn}-1}\,dg\,dX\,dY.
\]
Using Bruhat decomposition for $g$, the integral can be rewritten as
\begin{multline*}
\int_{\zerocol}\int_{\zerocol}\int_{N'_{\GLn}}\int_{T'_{\GLn}}
W\left(\toMd(t)w_{2{\rkn},{\rkn}}'\toM(n)\toU_{\GLnn}(X)\auxauxone\right)
\\ \d{W}  \left(\toMd(t)w_{2{\rkn},{\rkn}}'\toM(n)\toU_{\GLnn}(Y)\auxauxone\right)
\abs{\det t}^{{\rkn}-1}\modulus_{B'_{\GLn}}(t)^{-1} \,dt\,dn\,dX\,dY.
\end{multline*}
Since $\toU_\GLnn(\one_{1,1}+\zerocol)$ is invariant under conjugation by $\toM(N'_{\GLn})$,
by changing variables in $X$ and $Y$ we get
\begin{multline} \label{eq: reformfe}
A_{\rkn}(W,\d{W})=\int_{\zerocol}\int_{\zerocol}\int_{N'_{\GLn}}\int_{T'_{\GLn}}
W\left(\toMd(t)\auxone\toU_{\GLnn}(X)\toM(n)\right)\\ \d{W}  \left(\toMd(t)\auxone\toU_{\GLnn}(Y)\toM(n)\right)
\abs{\det t}^{{\rkn}-1}\modulus_{B'_{\GLn}}(t)^{-1}  \,dt\,dn\,dX\,dY
\end{multline}
where $\auxone=w_{2{\rkn},{\rkn}}'\auxauxone$.

\subsection{}\label{sec: fe}
In the first expression of \eqref{eq: otherBil} for $\newBil$, we have an inner integral of the form
$$\int_{N'_{\GLn}}\semipair{W_s\left(\cdot\toLevi(n)\right)}{\d{W}_{-s}}
\psi_{N'_{\GLn}}(n)  \,dn.$$
This leads us to apply the above functional equation in the following setting.
Let $\pi\in\Irr_{\temp}\GLnn$.
Define the bilinear form $\BilM(W,\d{W})$ on $\WhitM(\pi)\times\WhitMd(\d\pi)$ by
\begin{equation}\label{eq: defCW}
\BilM(W,\d{W}):=\int_{N'_{\GLn}}A_{\rkn}(\pi(\toM(n))W,\d{W})\psi_{N'_{\GLn}}(n)\,dn.
\end{equation}
It is shown in \cite[Appendix A]{MR3220931} that the integral is absolutely convergent.
\begin{remark}
For $(W,\d{W})\in \Ind(\WhitML(\pi))\times\Ind(\WhitMLd(\d\pi))$ we have
\begin{multline}\label{eq: relateBilM}
\BilM(\modulus_P^{-\frac12}W\circ\levi,\modulus_P^{-\frac12}\d{W}\circ\levi)=\BilM(\modulus_P^{-\frac12}W_s\circ\levi,
\modulus_P^{-\frac12}\d{W}_{-s}\circ\levi) \\
=\int_{N'_{\GLn}}\semipair{W_s\left(\cdot\toLevi(n)\right)}{\d{W}_{-s}}\psi_{N'_{\GLn}}(n)  \,dn.
\end{multline}
\end{remark}

For convenience, set
\begin{multline*}
\vrbar=\auxone(\toM(N'_{\GLn})\ltimes\toU_{\GLnn}(\zerocol))\auxone^{-1}=
w_{2\rkn,\rkn}'(\toM(N'_{\GLn})\ltimes\toU_{\GLnn}(\zerocol))w_{2\rkn,\rkn}'^{-1}
\\=\{\sm{I_{\rkn}}{}{x}{n^t}:x\in \zerocol, n\in N'_{\GLn}\}\subset N_{\GLnn}^t\cap\mira^*.
\end{multline*}

\begin{definition} \label{def: spclWM2}
Let $\WhitMd(\pi)_{\natural}$ be the linear subspace of $\WhitMd(\pi)$ consisting of $W$ such that
\[
W(\cdot \auxone)\rest_{\mira^*}\in \swrz(N_\GLnn\bs\mira^*, \psi_{N_\GLnn}^{-1}) \text{ and }
W(\cdot \auxone)\rest_{\toMd(T'_{\GLn})\ltimes \vrbar}\in \swrz(\toMd(T'_{\GLn})\ltimes \vrbar).
\]
\end{definition}
It is easy to see that this definition is independent of the choice of $\auxauxone$ (and correspondingly $\auxone$).

We note that $\WhitMd(\pi)_{\natural}$ is nonzero, since the restriction of $\WhitMd(\pi)$ to $\mira^*$ contains $\swrz(N_\GLnn\bs\mira^*, \psi_{N_\GLnn}^{-1})$
and $\toMd(T'_{\GLn})\ltimes \vrbar\subset B_{\GLnn}^t\cap\mira^*$.
On the other hand, even if we assume that $\pi$ is supercuspidal, $\WhitMd(\pi)_{\natural}$ is a proper subspace of $\WhitMd(\pi)$ since the set
$N_\GLnn\cdot(\toMd(T'_{\GLn})\ltimes \vrbar)$ is not closed.

Let $\vf=(\kgrp_\GLnn)^*$ and let $\psi_{\vf}$ be the character on $\vf$ given by $\psi_{\vf}(m)=\psi_{\kgrp_\GLnn}(m^*)$.
Note that $\vf$ consists of the elements of $N_{\GLnn}$ whose $\rkn+1$-th column vanishes above the diagonal.

\begin{proposition} \label{prop: feuse}
Let $\pi\in\Irr_{\temp}\GLnn$.
Then for any $W\in\WhitM(\pi)$ and $\d{W} \in \spclWdM$,
$\BilM(W,\d{W})$ is equal to the absolutely convergent integral
\begin{multline*}
\int_{\vrq}\int_{\vrq}\int_{T'_{\GLn}}W\left(\toMd(t)\auxone r_1\right)
\d{W} \left(\toMd(t)\auxone r_2\right)
\abs{\det t}^{{\rkn}-1}\modulus_{B'_{\GLn}}(t)^{-1}\psi_{\vf}(r_1^{-1}r_2)\\ \,dt\,dr_2\,dr_1
\end{multline*}
where the integrand is compactly supported in all variables.
\end{proposition}

\begin{proof}
From \eqref{eq: reformfe} and \eqref{eq: defCW}, we get that $\BilM(W,\d{W})$ is equal to
\begin{multline*} 
\int_{N'_{\GLn}}\big(\int_{\zerocol}\int_{\zerocol}\int_{N'_{\GLn}}\int_{T'_{\GLn}}
W\left(\toMd(t)\auxone\toU_{\GLnn}(X)\toM(n_2n_1)\right)\d{W}  \left(\toMd(t)\auxone\toU_{\GLnn}(Y)\toM(n_2)\right)\\
\abs{\det t}^{{\rkn}-1}\modulus_{B'_{\GLn}}(t)^{-1}  \,dt\,dn_2\,d X\,d Y\big)\psi_{N'_{\GLn}}(n_1)\ dn_1.
\end{multline*}
By the condition on $\d{W}$, the integrand is compactly supported in $t$, $Y$ and $n_2$.
Since $\toMd(t)$ normalizes $\vrbar$, the integrand is also compactly supported in $X$, $n_1$ in view of Lemma \ref{lem: cslu} below.
A change of variable $n_1\mapsto n_2^{-1}n_1$ gives the identity in the proposition.
\end{proof}

\begin{lemma} \label{lem: cslu}
The restriction to $N_\GLnn^t\cap(\mira\cup\mira^*)$ of any $W\in C^{\smth}(N_\GLnn\bs\GLnn,\psi_{N_\GLnn})$ is compactly supported.
\end{lemma}

\begin{proof}
This is standard. By passing to $W(\cdot^*)$, it is enough to consider the support in $N_\GLnn^t\cap\mira$.
Let $n=uak$ with $a=\diag(a_1,\ldots,a_{2{\rkn}})$ be the Iwasawa decomposition of $n\in N_\GLnn^t\cap\mira$.
Then $a_{2\rkn}=1$ and $\abs{n_{i,j}}\le\abs{a_i\dots a_{2\rkn}}$ for all $1\le j\le i\le 2\rkn$.
Thus, if $W(n)\ne0$ then $a_i$ are bounded in terms of $W$ and hence $n$ is bounded. The lemma follows.
\end{proof}

\subsection{}

We will apply Proposition~\ref{prop: feuse} to get a new expression for $\newBil(W, \d W,s)$. Before stating the result, we first introduce an integral
that will appear in the new expression.

As in \cite[\S7.4]{MR3431601}, define
$$\factor(t):=\abs{t_1}^{-\rkn}\modulus_B^{\frac12}(\levi(t)),\ \ \ t=\diag(t_1,\ldots,t_{2{\rkn}})\in T_\GLnn.$$
In particular when $t=\toMd(t')$ with $t'\in T'_{\GLn}$, we have $\factor(t)=\modulus_{B'}(\levi'(t))^{\frac12}$.

For now, let $\auxtwo$ be an arbitrary element in $N_\GLnn$. (We will fix it in the next section in order to make the formulas look nicer.)
Let $\soment=\toMd(T'_{\GLn})\times Z_\GLnn$.
For $W\in C^{\smth}(N\bs G,\psi_N)$, $t\in \soment$ let
\begin{multline}\label{eq: defoldhalf}
\newhalf(W,t):=\factor(t)^{-1}\int_{\bigvr}W\left(\levi(t\auxtwo\auxone)
w_Uv\right)\psi_{\kgrp}(v)^{-1}\,dv\\=
\factor(t)^{-1}\psi_{N_\GLnn}(t\auxtwo t^{-1})\int_{\bigvr}W\left(\levi(t\auxone)
w_Uv\right)\psi_{\kgrp}(v)^{-1}\,dv\\
=\factor(t)^{-1}\psi_{N_\GLnn}(t\auxtwo t^{-1})\int_U\int_{\vrbar}W(\levi(tv\auxone  )w_U u)
\psi_{\vf}(\auxone^{-1}v\auxone)^{-1}\psi_U(u)^{-1}\ dv\ du
\end{multline}
provided that the integral converges. (It is easy to check that the integrand above is invariant under $\toLevi(N'_{\GLn})$ when $t\in \soment$.)
Similarly define $\newhalfd(\alt{W},t)$ for $\alt W\in C^{\smth}(N\bs G,\psi_N^{-1})$.

Note that for any $t\in\soment$, $\newhalf(W,t)$ is $(\kgrp,\psi_{\kgrp})$-equivariant in $W$.
It is also clear from the definition that (whenever defined)
\begin{equation}\label{eq: invcenter}
\newhalf(W_s,tz)=\newhalf(W_s,t)\abs{\det z}^{s+\frac12}\omega_\pi(z),\ \ \ z\in Z_\GLnn,t\in \soment
\end{equation}
for $W\in\Ind(\WhitML(\pi))$ with $\pi\in \Irr_{\gen}\GLnn$ where $\omega_\pi$ is the central character of $\pi$.

\begin{definition} \label{def: spclW2}
Let $\spclWd{\pi}$ be the linear subspace of $\Ind(\WhitMLd(\pi)){\spcl}$ spanned by the functions which vanish outside
$Pw_UN$ and on the big cell are given by
\[
W(u'm w_U u)=\modulus_P^{\frac12}(m)W^M(m)\phi(u),\,\,m\in \Levi,\,u,u'\in U
\]
with $\phi\in\swrz(U)$ and $W^M\circ\levi\in \WhitMd(\pi)_{\natural}$.
\end{definition}
This space is clearly nonzero since $\WhitMd(\pi)_{\natural}$ is nonzero.

\begin{lemma} \label{lem: convE}
Let $\pi\in\Irr_{\gen}\Levi$. For $\Re s\gg1$ the integral \eqref{eq: defoldhalf} defining $\newhalf(W_s,t)$ converges for any
$W\in\Ind(\WhitML(\pi))$ and $t\in \soment$ uniformly for $(s,t)$ in a compact set.
Hence, $\newhalf(W_s,t)$ is holomorphic for $\Re s\gg1$ and continuous in $t$.
If $\alt{W}\in\Ind(\WhitMLd(\pi))\spcl$ then $\newhalfd(\alt{W}_s,t)$ is entire in $s$ and locally constant in $t$, uniformly in $s$.
If moreover $\alt{W}\in\spclWd{\pi}$ then $\newhalfd(\alt{W}_s,t)$ is compactly supported in $t\in\toMd(T'_{\GLn})$, uniformly in $s$.
\end{lemma}

\begin{proof}
For the first assertion, since $T_{\GLnn}$ normalizes the group $\vrbar$ and
$U$ is normalized by $\Levi$, we need to check the convergence, locally uniformly in $m\in\Levi$, of
\[
\int_{\vrbar}\int_U \abs{W_s\left(\levi(v)w_Uum\right)}\,du\,dv.
\]
This expression is clearly locally constant in $m$, so we can  assume $m=e$.
Since $\vrbar\subset N_\GLnn^t\cap\mira^*$, the integral above is a `partial integration' of the double integral \eqref{conv: column} in Lemma \ref{L:conv}
(with $g=e$), thus converges by Remark \ref{rem: partint}.

By Lemma \ref{lem: cslu}, when $\alt{W}\in\Ind(\WhitMLd(\pi))\spcl$,
the integrand in the last expression of \eqref{eq: defoldhalf} (with $\alt{W}_s$ in place of $ W$) is
compactly supported in $v$ and $u$, uniformly in $s$  and locally uniformly in $t$.
It is then clear that $\newhalfd(\alt{W}_s,t)$ is locally constant in $t$.
The second part follows. Moreover, if $\alt{W}\in\spclWd{\pi}$, it is clear from the definition of the latter space that
the integrand in \eqref{eq: defoldhalf} (with $\alt{W}$ in place of $ W$) is compactly supported in $v$ and $t\in \toMd(T'_{\GLn})$.
The lemma follows.
\end{proof}

\subsection{}

\begin{proposition} \label{prop: C by E}
Let $\pi\in\Irr_{\temp}\Levi$. Then for $\Re s\gg1$ we have
\begin{equation} \label{eq: BilE2}
\newBil(W,\d{W},s)=\int_{\toMd(T'_{\GLn})}\newhalf(W_s,t)\newhalfd(\d{W}_{-s},t)\,\frac{dt}{\abs{\det t}}
\end{equation}
for any $W\in\Ind(\WhitML(\pi))$ and $\d{W}\in \spclWd{\d\pi}$, where (by Lemma \ref{lem: convE})
the integrand on the right-hand side is continuous and compactly supported.
\end{proposition}

\begin{proof}
First note that the element $\auxtwo$ has no effect on the validity of \eqref{eq: BilE2}, so may ignore it from the consideration.
By Proposition \ref{prop: adjrltn}, if $\d{W}\in\Ind(\WhitMLd(\d\pi)){\spcl}$ the integrals defining $\newBil(W,\d{W},s)$ are absolutely convergent when $\Re s\gg1$.
If moreover $\d{W}\in \spclWd{\d\pi}$ then by the relation \eqref{eq: relateBilM}, (the first expression of) \eqref{eq: otherBil} and Proposition \ref{prop: feuse},
$\newBil(W,\d{W},s)$ is equal to
\begin{multline*}
\int_U\int_U\big(\int_{T'_{\GLn}}\int_{\vrq}\int_{\vrq}W_s\left(\levi(\toMd(t)\auxone r_1)w_Uu_1\right)\d{W}_{-s}\left(\levi(\toMd(t)\auxone r_2)w_Uu_2\right)\\
\modulus_{B'_{\GLn}}(t)^{-1}\modulus_P(\toLevid(t))^{-1}\abs{\det t}^{\rkn-1}\psi_U(u_1)^{-1}\psi_{\vf}(r_1)^{-1}\psi_U(u_2)
\psi_{\vf}(r_2)\,dr_1\,dr_2\,\ dt\big)\,du_1\,du_2.
\end{multline*}
Using the fact that for $t\in T'_{\GLn}$
\[
\modulus_{B'_{\GLn}}(t)^{-1}\modulus_P(\toLevid(t))^{-1}\abs{\det t}^{\rkn-1}=\modulus_{B'}(\levi'(t))^{-1}\abs{\det t}^{-1},
\]
to finish the proof of Proposition \ref{prop: C by E} it suffices to show the convergence of
\begin{multline*}
\int_{T'_{\GLn}}
\int_{\bigvr}\abs{W_s\left(\levi(\toMd(t)\auxone)w_Uv_1\right)}\,dv_1
\int_{\bigvr}\abs{\d{W}_{-s}\left(\levi(\toMd(t)\auxone)w_Uv_2\right)}\,dv_2\\\modulus_{B'}(\levi'(t))^{-1}\,\frac{dt}{\abs{\det t}}.
\end{multline*}
It follows from Lemma~\ref{lem: convE} that the integrals over $v_1$ and $v_2$ converge for a fixed $t$
and that the integral over $v_2$ is compactly supported as a function of $t$.
The argument in Lemma~\ref{lem: convE} also shows that the resulting integrals over $v_1$ and $v_2$ are smooth in $t$.
Thus, the integral converges.
\end{proof}

Combining Proposition \ref{prop: C by E} with Corollary \ref{cor: FEBil} and \eqref{eq: invcenter} we get
\begin{proposition}\label{prop: afterfe}
Let $\pi\in\Irr_{\temp}\Levi$. Then for $-\Re s\gg 1$ and any $W\in\spclW$, $\alt{W}\in\spclWd{\pi}$, we have
\begin{equation}\label{eq: BilE}
\Bil(W,M(s)\alt{W},s)=\int_{Z_\GLnn\bs \soment}\newhalf(M^*_sW,t)\newhalfd(\alt{W}_s,t)\,\frac{dt}{\abs{\det t}}
\end{equation}
where the integrand is continuous and compactly supported.
\end{proposition}

\section{Proof of Proposition \ref{prop: main2}} \label{sec: define E}

Proposition \ref{prop: afterfe} is an important (hard-earned) step towards the proof of Proposition \ref{prop: main2}.
However, it is still insufficient. First, it is not valid at the point $s=\frac12$ which is pertinent for the left-hand side of
\eqref{eq: main2}. Second, the functions $\newhalf$ have to be explicated in order to match the right-hand side of \eqref{eq: main2}.
Also, so far we have only used the fact that $\pi\in\Irr_{\meta}\GLnn$ to reduce to special sections, but
it should be clear from the nature of our formula that this fact has to enter the computation
in a more substantial and quantitative way. Fortunately, these issues were taken up in \cite[\S11]{MR3431601}.
We will recall these results and explain how they are used to conclude the proof of Proposition \ref{prop: main2} from Proposition~\ref{prop: afterfe}.

\subsection{}
Let $\ddg=\diag(1,-1,\ldots,(-1)^{\rkn-1})\in\Mat_\rkn$.
We now fix
\[
\auxtwo=\toU_{\GLnn}(-\frac12\ddg\wn)\in N_\GLnn.
\]
This element is denoted by $\epsilon'$ in \cite[p. 9550]{MR3431601} with the parameter $\num=-\frac12$
in the notation of that paper. We also fix $\auxauxone=\toU_\GLnn(\ddg)$ (and correspondingly
$\auxone=w_{2\rkn,\rkn}'\auxauxone$ -- see \S\ref{sec: bepsilon1}).
The reason for fixing the elements this way is (partly) due to the special form of the
character $\psi_{\bar U}$ on $\bar U$ given by
\begin{equation} \label{eq: indepchar}
\psi_{\bar U}(\bar v)=\psi_U((\levi(\auxtwo \auxone)w_U)^{-1}\bar v\levi(\auxtwo\auxone)w_U)^{-1}=
\psi(\bar v_{2{\rkn}+1,1}), \ \ \bar v\in \bar U.
\end{equation}

We rewrite the first expression on the right-hand side of \eqref{eq: defoldhalf} by splitting the integral into integrals over $U$ and
$\toLevi(N'_{\GLn})\bs\kgrp_\Levi$, making the change of variables
\[
v\mapsto (\levi(\auxtwo\auxone)w_U)^{-1}\bar v\levi(\auxtwo\auxone) w_U
\]
in the integral over $U$, and conjugating the variable in the second integral by $w_U$. We get (for $\Re s\gg 1$)
\begin{multline}\label{eq: defnewhalf}
\newhalf(W_s,t)=\factor(t)^{-1}\int_{\vrq}\int_{\bar U}
W_s\left(\levi(t)\bar v \levi( \auxtwo\auxone r) w_U\right)
\psi_{\bar U}(\bar v)\psi_{\vf}(r)^{-1}\,d\bar v\,dr\\=
\factor(t)^{-1}\int_{\vrbar}\int_{\bar U}
W_s\left(\levi(t)\bar v \levi( \auxtwo r\auxone) w_U\right)
\psi_{\bar U}(\bar v)\psi_{\vrbar}(r)\,d\bar v\,dr
\end{multline}
where $\psi_{\vrbar}(r)=\psi_{\vf}(\auxone^{-1}r\auxone)^{-1}$.

We now quote the first pertinent result from \cite[\S11]{MR3431601}.
Let $\spcltrs$ be the subtorus $\prod_{i=1}^{\rkn-1}\Imk_i$ of $\soment$  (of codimension two) where
$\Imk_i$ is the one-dimensional torus
\[
\Imk_i:=\{\diag(\overbrace{z^{-1},\ldots,z^{-1}}^{2\rkn-i},\overbrace{z,\ldots,z}^i):z\in F^*\}.
\]
For any $f\in\swrz(\spcltrs)$ and $g\in C(\soment)$ we write $f*g(\cdot)=\int_{\spcltrs} f(t)g(\cdot t)\ dt$.

\begin{theorem} \label{thm: s11mt}
For any $W\in C^{\smth}(N\bs G,\psi_N)$ which is left-invariant under a compact open subgroup of $Z_{\Levi}$
and any $f\in\swrz(\spcltrs)$, the function $f*\newhalf(W_s,t)$ extends to an entire function in $s$ which is locally constant in $t$, uniformly in $s$.
Moreover, if $\pi\in\Irr_{\meta,\temp}\GLnn$ then
\[
f*\newhalf(M^*_sW,t)\rest_{s=\frac12}=\eps_\pi^{\rkn} \Mint(M^*W)\int_{\spcltrs}f(t')\ dt'
\]
for any $W\in\Ind(\WhitML(\pi))$, $t\in\soment$ and $f\in\swrz(\spcltrs)$.
\end{theorem}

The first part is \cite[Corollary~11.9]{MR3431601} (except for a slight correction which we now explain).
Recall some notation from [ibid.].
As in \cite[\S6]{MR3431601} let $\zigzag$ be the unipotent subgroup of $\GLnn$ given by
\begin{align*}
\zigzag=\{m\in\GLnn:m_{i,i}=1\ \forall i,m_{i,j}=0\text{ if }&\text{either ($j>i$ and $i+j>2\rkn$)}\\
&\text{or ($i>j$ and $i+j\le 2\rkn+1$)}\}
\end{align*}
and let $\psi_{\zigzag}$ be its character
\begin{equation}\label{eq: defpsie}
\psi_{\zigzag}(m)=\psi(m_{1,2}+\ldots+m_{\rkn-1,\rkn}-m_{\rkn+2,\rkn+1}-\ldots-m_{2\rkn,2\rkn-1}),\ \ m\in\zigzag.
\end{equation}
The group $\levi(\zigzag)$ stabilizes the character $\psi_{\bar U}$.
Let $\Etwon=\levi(\zigzag)\ltimes\bar U$ with the character
\[
\psi_{\Etwon}(\levi(m)\bar u)=\psi_{\zigzag}(m)\psi_{\bar U}^{-1}(\bar u),\ \ m\in\zigzag, \bar u\in\bar U.
\]
Also, let $\zigzag^+=\zigzag\cap N_\GLnn$, $\vs=\zigzag\cap N_\GLnn^t$ and
\[
\remr=\{\toU_{\GLnn}(X)^t:X\in\zerocol,X_{i,j}=0\text{ if }i+j>\rkn+1\}.
\]
We have $\zigzag=\zigzag^+\cdot\vs$ and $\vrbar=\vs\cdot\remr$.
In \cite[\S10.4]{MR3431601} it was erroneously claimed that the second product is semi-direct.
Fortunately, this does not have any bearing on the argument. All what matters is that the character
$\psi_{\vrbar}$ is given by $\psi_{\vrbar}(vn)=\psi_{\vs}(v)$, $v\in\vs$, $n\in\remr$ where $\psi_{\vs}(r)$ is defined right after (7.11) in [ibid.].
Thus, the middle integral in [ibid., (11.11)] should be taken over $\vs\cdot\remr$, i.e., over $\vrbar$ in our notation.
(Also, the variable $v$ should be replaced by $r$.)
With this correction, the expression in [ibid., (11.11)], evaluated at $\levi(\auxone)w_U$, is $(f\factor)*\newhalf(W_s,t)$
and its analytic continuation is given by the expression in [ibid., (11.10)] which amounts to a finite sum.

Meanwhile, from the definition of $\Mint$ in \eqref{eq: defMint} we have
\[
\Mint(W)=\int_{\levi(V_{\GLnn}^*)\bs \gamma\Vm \gamma^{-1}} W(v\gamma\auxthree)\psi_{\Vm}((\gamma\auxthree)^{-1}v\gamma\auxthree)^{-1}\, dv.
\]
We can integrate instead over the group $\subUbar$ in \cite[\S11.5]{MR3431601}, consisting of
elements $\bar u$ in $\bar U$ such that $\bar u_{2\rkn+i,j}=0$ whenever $i\geq \rkn$ and $j\leq \rkn$.
The character $\oldsubUbar(\bar u):=\psi_{\Vm}((\gamma\auxthree)^{-1}\bar u\gamma\auxthree)^{-1}$ on $\subUbar$ also matches the one defined in
\cite[\S11.3, 11.5]{MR3431601} (with the parameter $\num=-\frac12$).
Thus, $\Mint(W)$ is ${\mathcal T}'(W)(\gamma\auxthree)$ in the notation of \cite[\S11.5]{MR3431601}.
The second part amounts to the first statement of \cite[Corollary~11.10]{MR3431601} upon taking
\[
\auxthree=(\few\gamma)^{-1}\levi(\auxone)w_U=\toU_\Levi((-1)^\rkn\ddg)
\]
where
\[
\few:=\toG(w'_{U'})\levi(w_{2{\rkn},{\rkn}}')=
\left(\begin{smallmatrix}&I_\rkn&&\\&&&w_\rkn\\-w_\rkn&&&\\&&I_\rkn&\end{smallmatrix}\right)
\]
is as in the bottom of p. 9523 of [ibid.].

Theorem \ref{thm: s11mt} is a crucial step.
Recall that the Langlands quotient of $\Ind(\WhitML(\pi),\frac12)$, i.e., the image under $M^*$,
admits a unique $(\kgrp,\psi_{\kgrp})$-equivariant functional.
Thus, $\newhalf(M^*W,t)$ (which is technically not defined) is a priori
proportional to $\Mint(M^*W)$ and the constant of proportionality is a function depending on $t$.
Theorem \ref{thm: s11mt} essentially says that in fact this function is a constant
which can be computed explictly.
The main input is that the Langlands quotient admits a realization in $C^{\smth}(H\bs G)$ where
$H\simeq\Sp_\rkn\times\Sp_\rkn$ is the centralizer of $\levi(E)$ in $G$.
(Recall that $E=\diag(1,-1,\dots,1,-1)$.) We refer to \cite{MR3431601} for more details.

We conclude

\begin{corollary} \label{cor: tindep}
Suppose that $\pi\in\Irr_{\meta,\temp}\GLnn$. Then for any $W\in\spclW$, $\alt{W}\in\spclWd{\pi}$ we have
\begin{equation}\label{eq: aftertindep}
\Bil(W,M(\frac12)\alt{W},\frac12)=\eps_\pi^{\rkn} \Mint(M^*W)\int_{Z_\GLnn\bs \soment} \newhalfd(\alt{W}_{\frac12},t)\,\frac{dt}{\abs{\det t}}.
\end{equation}
\end{corollary}

\begin{proof}
By Lemma~\ref{lem: convE}, for any $\alt{W}\in\spclWd{\pi}$ there exists $K_0\in \csgr(\soment)$
such that $\newhalfd(\alt{W}_s,\cdot)\in C(\soment)^{K_0}$ for all $s$ and
$\newhalfd(\alt{W}_s,\cdot)$ is compactly supported modulo $\Z_\GLnn$ uniformly in $s$.
Suppose that $f\in\swrz(\spcltrs)$ is supported in $\spcltrs\cap K_0$ and let $\d f(t):=f(t^{-1})$.
By Proposition~\ref{prop: afterfe}, for $-\Re s\gg1$ and $W\in\spclW$ we have
\begin{multline}\label{eq: introavg}
\Bil(W,M(s)\alt{W},s)\int_{\spcltrs}f(t)\ dt
=\int_{Z_\GLnn\bs \soment}\newhalf(M^*_sW,t)\d{f}*\newhalfd(\alt{W}_s,t)\,\frac{dt}{\abs{\det t}}
\\=\int_{Z_\GLnn\bs \soment}f*\newhalf(M^*_sW,t)\newhalfd(\alt{W}_s,t)\,\frac{dt}{\abs{\det t}}.
\end{multline}
From the first part of Theorem \ref{thm: s11mt} (i.e., \cite[Corollary~11.9]{MR3431601})
we infer that both sides of \eqref{eq: introavg} are meromorphic functions in $s$ and the identity holds whenever
$M(s)$ is holomorphic. Specializing to $s=\frac12$ and using the second part of Theorem \ref{thm: s11mt} (i.e., \cite[Corollary~11.10]{MR3431601})
the corollary follows.
\end{proof}

\subsection{}

It remains to compute the integral on the right-hand side of \eqref{eq: aftertindep}.
Again, this is essentially done in \cite{MR3431601} and it relies heavily on the fact that $\pi\in\Irr_{\meta}\GLnn$.

Let
\begin{multline*}
\WhitM(\pi)_{\natural\natural}=\{W\in\WhitM(\pi):
W\rest_{\mira^*}\in \swrz(N_\GLnn\bs\mira^*, \psi_{N_\GLnn})\text{ and }\\
W\rest_{\toMd(T'_{\GLn})\ltimes \zigzag}\in \swrz(\zigzag^+\bs\toMd(T'_{\GLn})\ltimes\zigzag,\psi_{\zigzag})\}.
\end{multline*}
Then $\WhitM(\pi)_{\natural\natural}$ is invariant under $\pi(\auxtwo)$ and
$\pi(u\auxone)W\in\WhitM(\pi)_{\natural\natural}$ for any $W\in\WhitM(\pi)_{\natural}$ and $u\in\remr$.
Note that $\auxtwo$ normalizes $\zigzag$, $\zigzag^+$ and $\psi_{\zigzag}$ (by direct computation).
Similarly $\auxone(w_{2\rkn,\rkn}')^{-1}$ stabilizes $(\zigzag,\psi_{\zigzag})$.

\begin{theorem} \label{thm: auxLM}
Let $\pi\in\Irr_{\meta,\temp}\GLnn$. Then
\begin{enumerate}
\item (see \cite[Proposition 3.2]{MR3431601} and \cite{MR3430877})
The integral
\[
\per{H_{\GLnn}}(W):=\int_{N_\GLnn\cap H_\GLnn\bs\mira\cap H_\GLnn} W(p)\ dp
\]
converges and defines a non-zero $H_{\GLnn}$-invariant functional on $\WhitM(\pi)$.

\item \label{part: lem11.6} (\cite[Proposition 11.1]{MR3431601}) For any $W\in\WhitM(\pi)_{\natural\natural}$
we have
\[
\int_{\zigzag^+\bs\zigzag}\int_{\toMd(T'_{\GLn})}\factor(t)^{-1}\abs{\det t}^{\rkn}W(tr)\psi_{\zigzag}(r)^{-1}\,dt\,dr
=\int_{\zigzag\cap H_\GLnn\bs\zigzag}\per{H_{\GLnn}}(\pi(n)W)\psi_{\zigzag}(n)^{-1}\,dn.
\]
\item (\cite[Lemma~4.3]{MR3431601}) \label{part: defLW} The integral
\begin{multline*}
L_W(g):=
\int_{P\cap H\bs H}\int_{N_\GLnn\cap H_\GLnn\bs\mira\cap H_\GLnn}W(\levi(p)hg)\abs{\det p}^{-(\rkn +1)}\ dp\ dh\\=
\int_{H\cap\bar U}\per{H_{\GLnn}}((\modulus_P^{-\frac12}I(\frac12,\bar ug)W)\circ\levi)\,d\bar u
\end{multline*}
converges for any $W\in\Ind(\WhitML(\pi),\frac12)$ and defines an intertwining map
\[
\Ind(\WhitML(\pi),\frac12)\rightarrow C^{\smth}(H\bs G).
\]
\item \label{part: mintL} (second statement of \cite[Corollary 11.10]{MR3431601}) We have
\begin{equation}\label{eq: Mint}
\Mint(M^*W)=\eps_\pi^{\rkn+1}\int_{\remr}\big(\int_{H\cap\Etwon\bs\Etwon}
L_W(v \levi(\auxtwo u\auxone)w_U)\psi_{\Etwon}^{-1}(v)\,dv\big)\ du.
\end{equation}
\end{enumerate}
\end{theorem}

The first part is based on a variant of Bernstein's theorem on $\mira$-invariant distributions \cite{MR748505}.
The second part is proved by ``root killing'' -- a method which goes back at least to \cite{MR519356} and has been used
constantly ever since. The third part is essentially a formal consequence of the first part. The last (and most complicated) part
is closely related to Theorem \ref{thm: s11mt} (with $\auxthree$ as before).
We refer to \cite{MR3431601} for more details.

\begin{remark}
While the above formula for $\Mint(M^*W)$ is convenient for computing the right-hand side of \eqref{eq: aftertindep},
it is also helpful to write the integration with the variables at the right. We will do that in \eqref{eq: Mint2} below.
\end{remark}
We obtain

\begin{corollary} \label{cor: Mfactor}
Let $\pi\in\Irr_{\meta,\temp}\GLnn$. Then for any $W\in\Ind(\WhitML(\pi))\spcl_\natural$ we have
\begin{equation} \label{eq: inttE}
\int_{Z_\GLnn\bs \soment}\newhalf(W_{\frac12},t)\,\frac{dt}{\abs{\det t}}=\eps_\pi^{{\rkn}+1}\Mint(M^* {W}).
\end{equation}
\end{corollary}

Note that it is not a priori clear that the left-hand side of \eqref{eq: inttE} factors through $M^*W$.

\begin{proof}
We may assume without loss of generality that
\begin{equation}\label{eq: defW2M}
W_{\frac12}(u'\levi(m)w_Uu)=W^{\GLnn}(m)\abs{\det m}^{\frac12}\modulus_P(\levi(m))^{\frac12}\phi(u), \ m\in \GLnn, u,u'\in U
\end{equation}
with $W^{\GLnn}\in\WhitM(\pi)_{\natural}$ and $\phi\in\swrz(U)$.
We evaluate the left-hand side $I$ of \eqref{eq: inttE} using the  last expression in \eqref{eq: defoldhalf}
(where we recall that the integrand is compactly supported by Lemma~\ref{lem: convE}).
Thus,
\[
I=I'\int_U\phi(v)\psi_U^{-1}(v)\,d v
\]
where
\[
I'= \int_{\toMd(T'_\GLn)}\int_{\vrbar}\factor(t)^{-1}\abs{\det t}^{\rkn}W^{\GLnn} ( t \auxtwo r\auxone)\psi_{\vrbar}(r)\,dr\,dt.
\]
The integrand in $I'$ is compactly supported because $W^{\GLnn}\in \WhitM(\pi)_{\natural}$.
Note for $r\in \vs=\zigzag\cap\vrbar$, $\psi_{\vrbar}(r)=\psi_{\zigzag}(r)^{-1}$.
We write
\begin{multline*}
I'= \int_{\remr}\big(\int_{\vs}\int_{\toMd(T'_\GLn)}\factor(t)^{-1}\abs{\det t}^{\rkn}
W^{\GLnn} (t\auxtwo ru\auxone)\psi_{\zigzag}(r)^{-1}\,dt\,dr\big)\,du
\\=\int_{\remr}\big(\int_{\zigzag^+\bs\zigzag}\int_{\toMd(T'_\GLn)}\factor(t)^{-1}\abs{\det t}^{\rkn}
W^{\GLnn} (t\auxtwo ru\auxone)\psi_{\zigzag}(r)^{-1}\,dt\,dr\big)\,du
\\=\int_{\remr}\big(\int_{\zigzag^+\bs\zigzag}\int_{\toMd(T'_\GLn)}\factor(t)^{-1}\abs{\det t}^{\rkn}
W^{\GLnn} (tr\auxtwo u\auxone)\psi_{\zigzag}(r)^{-1}\,dt\,dr\big)\,du
\end{multline*}
(since $\auxtwo$ stabilizes $\psi_{\zigzag}$).
For the double integral in the brackets we apply part \ref{part: lem11.6} of the theorem above to $\pi(\auxtwo u\auxone)W^{\GLnn}$
(which is applicable since $W^{\GLnn}\in\WhitM(\pi)_{\natural}$). We get
\[
I'=\int_{\remr}\big(\int_{\zigzag\cap H_\GLnn\bs\zigzag}\per{H_{\GLnn}}(\pi(n\auxtwo u\auxone)W^{\GLnn})\psi_{\zigzag}(n)^{-1}\,dn\big)\ du.
\]
Thus,
\[
I=\int_{\remr}\big(\int_{\zigzag\cap H_\GLnn\bs\zigzag} \int_U
\per{H_{\GLnn}}(\pi(n\auxtwo u\auxone)W^{\GLnn})\psi_{\zigzag}(n)^{-1}\phi(v)\psi_U^{-1}(v)\,d v\,dn\big)\,du.
\]
From \eqref{eq: defW2M}, $(\modulus_P^{-\frac12}I(\frac12, \levi(m)w_Uv)W)\circ \levi=
\phi(v)\modulus_P^{\frac12}(\levi(m))\abs{\det m}^\frac12\pi(m)W^{\GLnn}$
for any $v\in U$, $m\in\GLnn$. Thus, $I$ equals
\[
\int_{\remr}\big(\int_{\zigzag\cap H_\GLnn\bs\zigzag}\int_U
\per{H_{\GLnn}}((\modulus_P^{-\frac12}I(\frac12, \levi(n\auxtwo u\auxone)w_Uv)W)\circ \levi)
\psi_{\zigzag}(n)^{-1}\psi_U(v)^{-1}\,dv\,dn\big)\,du.
\]
Since $\auxone^{-1}\remr\auxone\subset \vf$, the group $\levi(\auxone^{-1}\remr\auxone)$ stabilizes the character $\psi_U(w_U^{-1}\cdot w_U)$ on $\bar U$.
Making a change of variable
\[
v\mapsto (\levi(n\auxtwo u\auxone)w_U)^{-1}\bar v\levi(n\auxtwo u\auxone)w_U
\]
on $U$ we obtain
\[
I=\int_{ \remr}\big(\int_{\levi(\zigzag\cap H_\GLnn)\bs \Etwon}
\per{H_{\GLnn}}((\modulus_P^{-\frac12}I(\frac12,v\levi\left(\auxtwo u\auxone\right)w_U)W)\circ\levi)
\psi_{\Etwon}^{-1}(v)\,dv\big)\,du.
\]
Integrating over $H\cap\bar U$ first, we get from part \ref{part: defLW} of the theorem that
\[
I=\int_{\remr}\big(\int_{H\cap\Etwon\bs\Etwon}L_W(v\levi(\auxtwo u\auxone) w_U)\psi_{\Etwon}^{-1}(v)\,dv\big)\,du.
\]
Finally, by the last part of the theorem, this is equal to $\eps_\pi^{{\rkn}+1}\Mint(M^*W)$ as required.
\end{proof}

\subsection{}

\begin{proof}[Proof of Proposition~\ref{prop: main2}]
From Corollaries \ref{cor: tindep} and \ref{cor: Mfactor} we get that \eqref{eq: main2} holds for
$(W,\alt{W})\in\spclW\times\spclWd{\pi}$, namely
\[
\Bil(W,M(\frac12)\alt{W},\frac12)=\eps_\pi \Mint(M^*W)\Mintd(M^*\alt{W}).
\]
On the other hand, as in the proof of Lemma \ref{lem: Bnonvanish}, it follows from the definition \eqref{eq: defoldhalf} and the fact that
the image of the restriction map
$\WhitMLd(\pi)\rightarrow C(N_\Levi\bs\levi(\mira^*),\psi_{N_\Levi}^{-1})$ contains
$\swrz(N_\Levi\bs\levi(\mira^*),\psi_{N_\Levi}^{-1})$
that the linear map $\spclWd{\pi}\rightarrow \swrz(T'_{\GLn})$
given by $\alt{W}\mapsto\newhalfd(\alt{W}_{\frac12},\cdot)$ is onto.
Therefore, by Corollary \ref{cor: Mfactor} the linear form $\Mintd(M^*\alt{W})$ is nonvanishing on $\spclWd{\pi}$.
By Corollary \ref{cor: Bnonvanish} we conclude that \eqref{eq: main2} holds for
all $\alt{W}$ (not necessarily in $\spclWd{\pi}$). Proposition~\ref{prop: main2} follows.
\end{proof}

By the discussion before Proposition~\ref{prop: main2}, this concludes the proof of Theorem \ref{thm: main} and thus the proof of Theorem~\ref{thm: intromain}.

\subsection{Central character of the descent}
Assume that $\pi\in\Irr_{\meta,\temp}\GLnn$.
Let us give another expression for $\Mint(M^*W)$, which we will use to prove a formula for the so-called ``central sign'' of the descent $\des(\pi)$.

Denote by $\auxfour$ the element $\levi(\auxtwo \auxone )w_U$. A change of variables in $u$ and $v$ in \eqref{eq: Mint} gives
\begin{equation}\label{eq: Mint2}
\Mint(M^*W)=\eps_\pi^{\rkn+1}\int_{\remr^{\sharp}}\big(\int_{(H\cap\Etwon)^{\auxfour}\bs\Etwon^{\auxfour}}
L_W(\auxfour v \levi(u))\psi_{\Etwon^{\auxfour}}^{-1}(v)\,dv\big)\ du.
\end{equation}
Here $\remr^{\sharp}=\{\toU_\GLnn(X):X_{i,j}=0 \text{ if either $j>i$ or $i=\rkn$}\}$,
$\Etwon^{\auxfour}=\auxfour^{-1}\Etwon \auxfour=\levi(\zigzag^{\sharp})\ltimes  U$, with
\[
\zigzag^{\sharp}=\{\sm{n_1}{v_1}{v_2}{n_2}: n_1,n_2\in N'_{\GLn}, I_\rkn+v_1,I_\rkn+v_2\in N'_{\GLn}\},
\]
and $\psi_{\Etwon^{\auxfour}}(\levi(m)u)=\psi_{\zigzag^{\sharp}}(m)\psi_U(u)$ with $m\in \zigzag^{\sharp}$, $u\in U$ and
\[
\psi_{\zigzag^{\sharp}}(m)=\psi(-m_{1,2}-\ldots-m_{\rkn-1,\rkn}+m_{\rkn+1,\rkn+2}+\ldots+m_{2\rkn-1,2\rkn}),\ \ m\in\zigzag^{\sharp}.
\]

From the above expression we get:
\begin{proposition}\label{prop: central}
Let $\pi\in \Irr_{\gen,\meta}\GLnn$, and $\tilde \pi=\des(\pi)$. Then $\tilde\pi(\animg{-I_{2\rkn}})=\gamma_{\psi^{-1}}((-1)^{\rkn})\eps_\pi$.
\end{proposition}

Recall that $\tilde\pi(\animg{-I_{2\rkn}})/\gamma_{\psi^{-1}}((-1)^{\rkn})$ is the central sign of $\tilde\pi$ (with respect to $\psi^{-1}$)
introduced by Gan--Savin \cite{MR2999299}.
(However, we do not assume $\tilde\pi$ is irreducible.)

\begin{proof}
We need to show $\nwhitform(M^*W,\animg{-I_{2\rkn}})=\gamma_{\psi^{-1}}((-1)^{\rkn})\eps_\pi\Mint(M^*W)$. For the moment assume that $\pi$ is tempered.
By Lemmas~\ref{lem: equnwhit} and \eqref{eq: Mint2}, $\nwhitform(M^*W,\animg{-I_{2\rkn}})$ equals
\[
\gamma_{\psi^{-1}}((-1)^{\rkn})\eps_\pi^{\rkn+1}\int_{\remr^{\sharp}}\big(\int_{(H\cap\Etwon)^{\auxfour}\bs\Etwon^{\auxfour}}
L_W(\auxfour v \levi(u)\toG(-I_{2\rkn}))\psi_{\Etwon^{\auxfour}}^{-1}(v)\,dv\big)\ du.
\]
Conjugating $v\levi(u)$ by $\toG(-I_{2\rkn})$ we obtain
\[
\gamma_{\psi^{-1}}((-1)^{\rkn})\eps_\pi^{\rkn+1}\int_{\remr^{\sharp}}\big(\int_{(H\cap\Etwon)^{\auxfour}\bs\Etwon^{\auxfour}}
L_W(\auxfour\toG(-I_{2\rkn}) v \levi(u))\psi_{\Etwon^{\auxfour}}^{-1}(v)\,dv\big)\ du.
\]
Here we need to verify that $\toG(-I_{2\rkn})$ stabilizes $(\Etwon^{\auxfour},\psi_{\Etwon^{\auxfour}})$ (which is clear) and normalizes $H^{\auxfour}$.
The second fact follows from the observation that $\auxfour \toG(-I_{2\rkn})=\levi({\bf a})\auxfour$ where
\[
{\bf a}=\diag(\frac12\ddg,2\ddg^*)\wnn\in H_{\GLnn}\wnn.
\]
Since $L_W(\levi(\wnn)\cdot)=\eps_\pi L_W(\cdot)$ \cite[(4.8)]{MR3431601}, we get
$$
\nwhitform(M^*W,\animg{-I_{2\rkn}})=\gamma_{\psi^{-1}}((-1)^{\rkn})\eps_\pi^{\rkn}\int_{\remr^{\sharp}}\big(\int_{(H\cap\Etwon)^{\auxfour}\bs\Etwon^{\auxfour}}
L_W(\auxfour v \levi(u))\psi_{\Etwon^{\auxfour}}^{-1}(v)\,dv\big)\ du.
$$
The claim now follows from the comparison with \eqref{eq: Mint2}.
Finally, the same argument as in \S\ref{sec: firstred} using the classification result Theorem~\ref{thm: genmetaclassification}
gives the proposition for all $\pi\in \Irr_{\gen,\meta}\GLnn$.
\end{proof}

\appendix

\section{Non-vanishing of Bessel functions} \label{sec: nonvanishing}
Let $G$ be a split group over a $p$-adic field.\footnote{The notation in the appendix is different from the body of the paper.}
Let $B=A\ltimes N$ be a Borel subgroup of $G$.
Let $G\spcl=Bw_0B$ be the open Bruhat cell where $w_0$ is the longest element of the Weyl group.
Fix a non-degenerate continuous character $\psi_N$ of $N$.
For any $\pi\in\Irr_{\gen,\psi_N}(G)$, the Bessel function $\Bes_\pi=\Bes_\pi^{\psi_N}$ of $\pi$ with respect to $\psi_N$
is the locally constant function on $G\spcl$ given by the relation
\begin{equation}\label{eq: defbes}
\stint_N W(gn)\psi_N(n)^{-1}\ dn=\Bes_\pi(g)W(e)
\end{equation}
for any $W\in\Whit^{\psi_N}(\pi)$ (see \cite{MR3021791}).
In this section we prove the following result.
\begin{theorem} \label{thm: nonvanishingBesselfunction}
For any tempered $\pi\in\Irr_{\gen,\psi_N}(G)$ the function $\Bes_\pi$ is not identically zero on $G\spcl$.
\end{theorem}

The argument is similar to the one in \cite{IZ}.

Fix a tempered $\pi\in\Irr_{\gen,\psi_N}(G)$ and realize it on its Whittaker model $\Whit^{\psi_N}(\pi)$.
Similarly, realize $\pi^\vee$ on $\Whit^{\psi_N^{-1}}(\pi^\vee)$. Thus, we get a pairing $(\cdot,\cdot)$ on
$\Whit^{\psi_N}(\pi)\times\Whit^{\psi_N^{-1}}(\pi^\vee)$.

Fix $\d{W}_0\in\Whit^{\psi_N^{-1}}(\pi^\vee)$ such that
\[
\stint_N(\pi(n)W,\d{W}_0)\psi_N(n)^{-1}\ dn=W(e)
\]
for all $W\in\Whit^{\psi_N}(\pi)$.
This is possible by \cite[Propositions~2.3, 2.10]{MR3267120}.
Then for any $g\in G$ we have
\begin{equation} \label{eq: W0vee}
W(g)=\stint_N(\pi(ng)W,\d{W}_0)\psi_N(n)^{-1}\ dn.
\end{equation}
Similarly, fix $W_0\in\Whit^{\psi_N}(\pi)$ such that
\begin{equation} \label{eq: W0}
\stint_N(W_0,\pi^\vee(n)W^\vee)\psi_N(n)\ dn=W^\vee(e)
\end{equation}
for all $W^\vee\in\Whit^{\psi_N^{-1}}(\pi^\vee)$.

Set $\Phi(g)=(\pi(g)W_0,W_0^\vee)$. Let $N^{\der}$ be the derived group of $N$.
Also let $\Xi$ be the Harish-Chandra function on $G$ (see e.g.~\cite{MR1989693}).

\begin{lemma} \label{lem: locallyl1}
The function
\[
g\mapsto\int_{N^{\der}}\int_{N^{\der}}\Phi(n_1gn_2)\ dn_1\ dn_2
\]
on $G\spcl$ in locally $L^1$ on $G$.
Moreover,
\[
g\mapsto\int_{N^{\der}}\int_{N^{\der}}\Xi(n_1gn_2)\ dn_1\ dn_2
\]
is locally $L^1$ on $G$.
\end{lemma}

\begin{proof}
The argument is exactly as in \cite[Lemma A.4]{IZ} using the convergence of
$\int_{N^{\der}}\Xi(n)\ dn$ \cite[Lemma 6.3.1]{1203.0039}.
\end{proof}

\begin{remark}\label{r: appendB}
Note that $g\mapsto\int_{N^{\der}}\int_{N^{\der}}\Xi(n_1gn_2)\ dn_1\ dn_2$
is locally constant on $G\spcl$. Thus, its local integrability on $G$
implies its convergence for any $g\in G\spcl$.
\end{remark}

For any $f\in \swrz(G)$ let $L_f(W)=\int_G f(g)W(g)\ dg$, $W\in\Whit^{\psi_N}(\pi)$.
Then $L_f\in\pi^\vee$. Let $L_f^*$ be the corresponding element in $\Whit^{\psi_N^{-1}}(\pi^\vee)$
and set $B_\pi(f)=L_f^*(e)$.
The distribution $f\mapsto B_\pi(f)$ is called the \emph{Bessel distribution}.
It is non-zero: we can choose $f\in \swrz(G)$ such that $L_f$ is non-trivial,
and then, by translating $f$ if necessary we can arrange that $L_f^*(e)\ne0$.

Note that by \eqref{eq: W0},
\[
B_\pi(f)=\stint_N(W_0,\pi^\vee(n)L_f^*)\psi_N(n)\ dn=\stint_N\big(\int_G f(g)W_0(gn^{-1})\ dg\big)\psi_N(n)\ dn
\]
and therefore by \eqref{eq: W0vee} we have
\begin{equation} \label{eq: firstBessel}
B_\pi(f)=\stint_N\big(\int_G f(g)\big(\stint_N \Phi(n_1gn_2^{-1})\psi_N(n_1)^{-1}\ dn_1\big)\ dg\big)\ \psi_N(n_2)\ dn_2.
\end{equation}

Let $A^1$ be the maximal compact subgroup of $A$.
Fix $\Omega_0\in\csgr(N)$ which is invariant under conjugation by $A^1$ (e.g., take $\Omega_0=N\cap K$).
Fix an element $a\in A$ such that $\abs{\alpha(a)}>1$ for all $\alpha\in\Delta_0$.
Consider the sequence $\Omega_n=a^n\Omega_0a^{-n}\in\csgr(N)$, $n=1,2,\dots$.
Any $\Omega_n$ is invariant under conjugation by $A^1$,
$\Omega_1\subset\Omega_2\subset\dots$ and $\cup\Omega_n=N$.

Let $A^d=A\cap G^{\der}$. Consider the family
\[
A_n=\{t\in A^d:\abs{\alpha(t)-1}\le q^{-n}\text{ for all }\alpha\in\Delta_0\}\in\csgr(A^d)
\]
which forms a basis of neighborhoods of $1$ for $A^d$.
We only consider $n$ sufficiently large so that the image of $A_n$ under the homomorphism $t\in A\mapsto(\alpha(t))_{\alpha\in\Delta_0}\in (F^*)^{\Delta_0}$
is $\prod_{\alpha\in\Delta_0}(1+\varpi^n\OO)$ where $\varpi$ is a uniformizer of $\OO$.

For any $\alpha\in\Delta_0$ choose a parameterization $x_\alpha:F\rightarrow N_\alpha$ such that $\psi_N\circ x_\alpha$ is trivial on $\OO$
but not on $\varpi^{-1}\OO$.
Let $N_n$ be the group generated by $ \langle x_\alpha(\varpi^{-n}\OO),\alpha\in\Delta_0\rangle$
and the derived group $N^{\der}$ of $N$. The following lemma is clear:

\begin{lemma} \label{lem: Anave}
For any $u\in N$ we have
\[
(\vol A_n)^{-1}\int_{A_n}\psi_N(tut^{-1})\ dt=\begin{cases}\psi_N(u)&u\in N_n,\\
0&\text{otherwise.}\end{cases}
\]
\end{lemma}

Next we prove:

\begin{lemma} \label{lem: besselAn}
Suppose that $f$ and $\Phi$ are bi-invariant under $A_n$.
Then we have
\[
B_\pi(f)=\int_Gf(g)\alpha_n(g)\ dg
\]
where
\[
\alpha_n(g)=\int_{N_n}\int_{N_n}\Phi(n_1gn_2^{-1})\psi_N(n_1)^{-1}\psi_N(n_2)\ dn_2\ dn_1.
\]
\end{lemma}

\begin{remark}
Of course we cannot conclude from the lemma by itself that $B_\pi$ is given by a locally $L^1$ function.
(This is conjectured to be the case.)
\end{remark}

\begin{proof}
We start with \eqref{eq: firstBessel}. Let $m$ be such that
\[
B_\pi(f)=\int_{\Omega_m}\int_G\big(\stint_N f(g)\Phi(n_1gn_2^{-1})\psi_N(n_1)^{-1}\ dn_1\big)\ dg\ \psi_N(n_2)\ dn_2.
\]
Then for any $t\in A_n$ we have
\begin{multline*}
B_\pi(f)=\int_{\Omega_m}\int_G\big(\stint_N f(gt)\Phi(n_1gn_2^{-1}t)\psi_N(n_1)^{-1}\ dn_1\big)\ dg\ \psi_N(n_2)\ dn_2\\=
\int_{\Omega_m}\int_G\big(\stint_N f(gt)\Phi(n_1gtn_2^{-1})\psi_N(n_1)^{-1}\ dn_1\big)\ dg\ \psi_N(tn_2t^{-1})\ dn_2\\=
\int_{\Omega_m}\int_G\big(\stint_N f(g)\Phi(n_1gn_2^{-1})\psi_N(n_1)^{-1}\ dn_1\big)\ dg\ \psi_N(tn_2t^{-1})\ dn_2.
\end{multline*}
Averaging over $t\in A_n$ we get
\[
B_\pi(f)=\int_{N_n\cap\Omega_m}\int_G\big(\stint_N f(g)\Phi(n_1gn_2^{-1})\psi_N(n_1)^{-1}\ dn_1\big)\ dg\  \psi_N(n_2)\ dn_2
\]
by Lemma \ref{lem: Anave}. Thus, for $m'$ sufficiently large (depending on $f$ and $m$) we have
\[
B_\pi(f)=\int_{N_n\cap\Omega_m}\int_G\int_{\Omega_{m'}}f(g)\Phi(n_1gn_2^{-1})\psi_N(n_1)^{-1}\ dn_1\ dg\  \psi_N(n_2)\ dn_2.
\]
Once again, for any $t\in A_n$
\begin{multline*}
B_\pi(f)=\int_{N_n\cap\Omega_m}\int_G\int_{\Omega_{m'}} f(tg)\Phi(tn_1gn_2^{-1})\psi_N(n_1)^{-1}\ dn_1\ dg\ \psi_N(n_2)\ dn_2\\=
\int_{N_n\cap\Omega_m}\int_G\int_{\Omega_{m'}} f(tg)\Phi(n_1tgn_2^{-1})\psi_N(t^{-1}n_1t)^{-1}\ dn_1\ dg\ \psi_N(n_2)\ dn_2\\=
\int_{N_n\cap\Omega_m}\int_G\int_{\Omega_{m'}} f(g)\Phi(n_1gn_2^{-1})\psi_N(t^{-1}n_1t)^{-1}\ dn_1\ dg\ \psi_N(n_2)\ dn_2.
\end{multline*}
As before, averaging over $A_n$ we get
\[
B_\pi(f)=\int_{N_n\cap\Omega_m}\int_G\int_{N_n\cap\Omega_{m'}} f(g)\Phi(n_1gn_2^{-1})\psi_N(n_1)^{-1}\ dn_1\ dg\ \psi_N(n_2)\ dn_2.
\]
Since $m$ and $m'$ can be chosen arbitrarily large, the lemma follows from the convergence of
\[
\int_{N_n}\int_G\int_{N_n}\abs{f(g)\Phi(n_1gn_2^{-1})}\ dn_1\ dg\ dn_2,
\]
i.e., Lemma \ref{lem: locallyl1}.
\end{proof}

Analogously, we prove:

\begin{lemma} \label{lem: besAn}
For $g\in G\spcl$, let $\Bes_\pi^{A_n}(g):=\vol(A_n)^{-2}\int_{A_n}\int_{A_n}\Bes_\pi(t_1gt_2)\ dt_1\ dt_2$.
Suppose that $\Phi$ is bi-invariant under $A_n$. Then
\[
\Bes_\pi^{A_n}(g)W_0(e)=\alpha_n(g)
\]
for all $g\in G\spcl$.
\end{lemma}

\begin{proof}
First note that for any compact subset $C$ of $G\spcl$ we have
\[
\int_{\Omega_m}\big(\stint_N\Phi(n_1gn_2^{-1})\psi_N(n_1)^{-1}\ dn_1\big)\psi_N(n_2)\ dn_2=\Bes_\pi(g)W_0(e)
\]
for all $g\in C$ and $m$ sufficiently large. Indeed, by \eqref{eq: W0vee} the inner stable integral is $W_0(gn_2^{-1})$.
Thus, the relation above follows from \cite{MR3021791} and \eqref{eq: defbes}.

Now, for any $t\in A_n$ we have
\begin{multline*}
\Bes_\pi(gt)W_0(e)=
\int_{\Omega_m}\big(\stint_N\Phi(n_1gtn_2^{-1})\psi_N(n_1)^{-1}\ dn_1\big)\psi_N(n_2)\ dn_2\\=
\int_{\Omega_m}\big(\stint_N\Phi(n_1gn_2^{-1}t)\psi_N(n_1)^{-1}\ dn_1\big)\psi_N(t^{-1}n_2t)\ dn_2\\=
\int_{\Omega_m}\big(\stint_N\Phi(n_1gn_2^{-1})\psi_N(n_1)^{-1}\ dn_1\big)\psi_N(t^{-1}n_2t)\ dn_2.
\end{multline*}
Thus, by Lemma~\ref{lem: Anave}, we have
\[
\vol(A_n)^{-1}\int_{A_n}\Bes_\pi(gt)\ dt\ W_0(e)=
\int_{N_n\cap\Omega_m}\big(\stint_N\Phi(n_1gn_2^{-1})\psi_N(n_1)^{-1}\ dn_1\big)\psi_N(n_2)\ dn_2.
\]
For $g$ in a compact set $C$ of $G\spcl$ and for $m'$ sufficiently large (depending on $C$ and $m$)
we can write this as
\[
\int_{N_n\cap\Omega_m}\int_{\Omega_{m'}}\Phi(n_1gn_2^{-1})\psi_N(n_1)^{-1}\ dn_1\ \psi_N(n_2)\ dn_2.
\]
Now, for any $t_1\in A_n$
\begin{multline*}
\vol(A_n)^{-1}\int_{A_n}\Bes_\pi(t_1gt_2)\ dt_2\ W_0(e)\\=
\int_{N_n\cap\Omega_m}\int_{\Omega_{m'}}\Phi(n_1t_1gn_2^{-1})\psi_N(n_1)^{-1}\ dn_1\ \psi_N(n_2)\ dn_2\\=
\int_{N_n\cap\Omega_m}\int_{\Omega_{m'}}\Phi(t_1n_1gn_2^{-1})\psi_N(t_1n_1t_1^{-1})^{-1}\ dn_1\ \psi_N(n_2)\ dn_2\\=
\int_{N_n\cap\Omega_m}\int_{\Omega_{m'}}\Phi(n_1gn_2^{-1})\psi_N(t_1n_1t_1^{-1})^{-1}\ dn_1\ \psi_N(n_2)\ dn_2.
\end{multline*}
Averaging over $t_1\in A_n$ we get
\[
\Bes_\pi^{A_n}(g)W_0(e)=\int_{N_n\cap\Omega_m}\int_{N_n\cap\Omega_{m'}}\Phi(n_1gn_2^{-1})\psi_N(n_1)^{-1}\psi_N(n_2)\ dn_1\ dn_2.
\]
The lemma follows from the convergence of
\[
\int_{N_n}\int_{N_n}\abs{\Phi(n_1gn_2^{-1})}\ dn_1\ dn_2,
\]
i.e., Remark~\ref{r: appendB}.
\end{proof}

\begin{proof}[Proof of Theorem \ref{thm: nonvanishingBesselfunction}]
Since $B_\pi(f)\ne0$ we have $\alpha_n\rest_{G\spcl}\not\equiv0$ for $n\gg1$ by Lemma \ref{lem: besselAn}.
Hence, by Lemma \ref{lem: besAn}, $\Bes_\pi^{A_n}\rest_{G\spcl}\not\equiv0$ and therefore $\Bes_\pi\big|_{G\spcl}\not\equiv0$.
\end{proof}

\begin{remark}
Theorem \ref{thm: nonvanishingBesselfunction} and its proof are also valid for $\Mp_n$.
\end{remark}


\providecommand{\bysame}{\leavevmode\hbox to3em{\hrulefill}\thinspace}
\providecommand{\MR}{\relax\ifhmode\unskip\space\fi MR }
\providecommand{\MRhref}[2]{%
  \href{http://www.ams.org/mathscinet-getitem?mr=#1}{#2}
}
\providecommand{\href}[2]{#2}

\end{document}